\documentclass[amsmath,amssymb,amsthm]{amsart}

\usepackage{graphicx}
\usepackage{dcolumn}
\usepackage{bm}
\usepackage{tikz}
\usetikzlibrary{calc}
\usepackage[all]{xy}
\usetikzlibrary{arrows,shapes}
\usepackage{url}

\newtheorem{theorem}{Theorem}[section]
\newtheorem{lemma}[theorem]{Lemma}
\newtheorem{proposition}[theorem]{Proposition}
\newtheorem{corollary}[theorem]{Corollary}

\theoremstyle{definition}
\newtheorem{definition}[theorem]{Definition}

\theoremstyle{remark}

\numberwithin{equation}{section}

\begin{document}

\title{Topology in cyber research}

\author{Steve Huntsman}
\address{BAE Systems FAST Labs}

\author{Jimmy Palladino}
\address{American University}

\author{Michael Robinson}
\address{American University}

%
%
%

\begin{abstract}
We give an idiosyncratic overview of applications of topology to cyber research, spanning the analysis of variables/assignments and control flow in computer programs, a brief sketch of topological data analysis in one dimension, and the use of sheaves to analyze wireless networks. 

The text is from a chapter in the forthcoming book \emph{Mathematics in Cyber Research} to be published by Taylor \& Francis.
\end{abstract}

\maketitle

Basic topological notions of connectivity are at the center of the cyber domain. Although graph/network theory addresses many problems relating to connectivity and global or qualitative structure in computer science and cybersecurity using techniques that trace their lineage to Euler (Figure \ref{fig:Konigsberg}), we sketch several ways in which distinctly modern topological approaches can help.  Taking connectivity as the base case, topological methods provide finer invariants that are useful for addressing more complex cyber problems. We review various relevant topological constructions, focusing on discrete structures that are naturally suited for addressing cyber-oriented problems.

\begin{figure}[htb]
\includegraphics[trim = 0mm 0mm 0mm 0mm, clip, keepaspectratio, width=.49\textwidth]{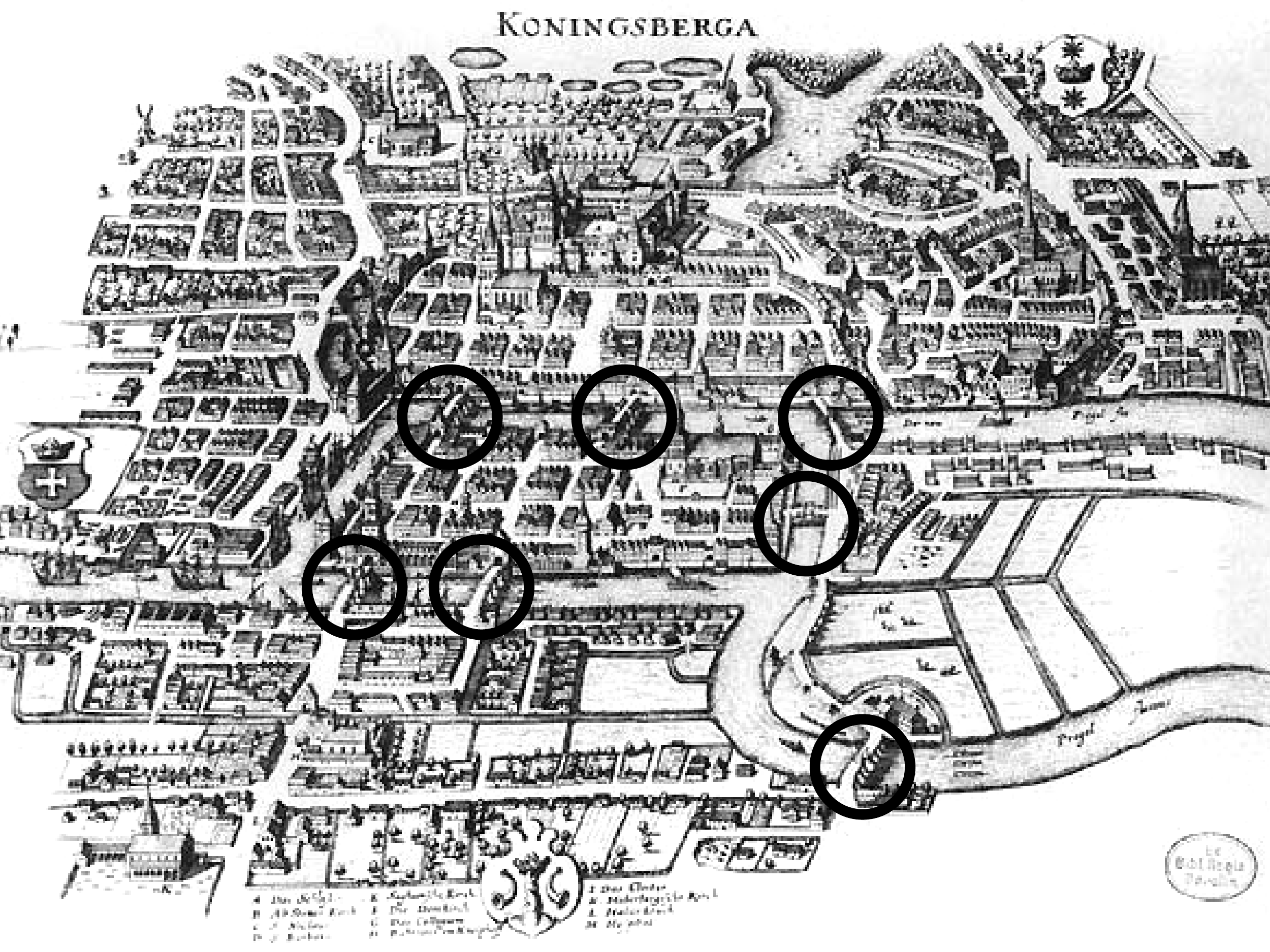}
\includegraphics[trim = 0mm 0mm 0mm 0mm, clip, keepaspectratio, width=.49\textwidth]{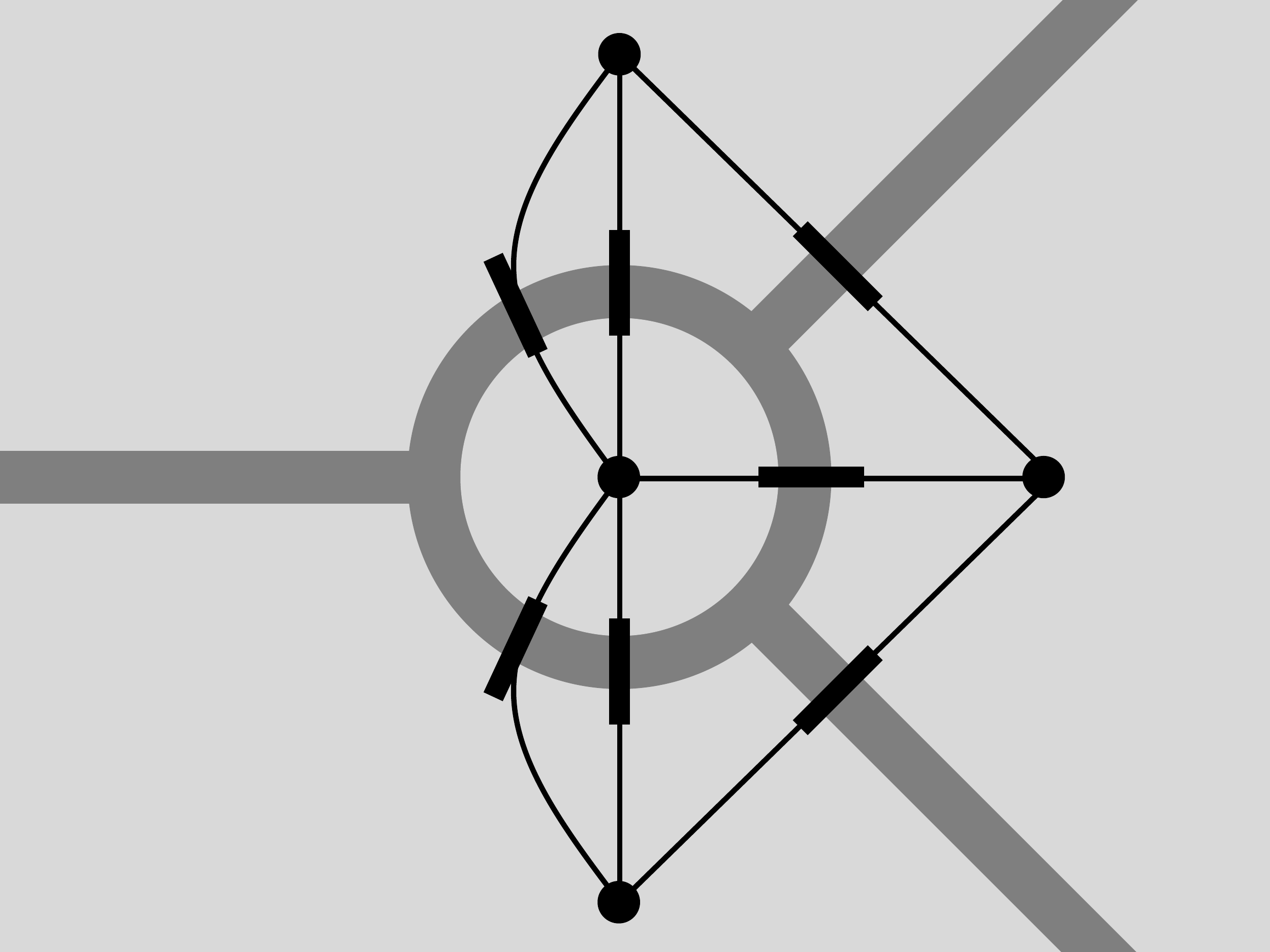}
\caption[The seven bridges of K\"onigsberg]{(L) The seven bridges of K\"onigsberg, indicated by circles. (R) A graphical representation of the bridges. Euler asked (and answered) the foundational problem in both graph theory and topology: is there a round trip that crosses every bridge exactly once? The answer is \emph{no}, because no orientation of the edges can give the same in- and out-degree to vertices with odd degree. The topological insight of Euler was that a combinatorial structure can faithfully represent connectivity properties of continuous bodies.}
\label{fig:Konigsberg}
\end{figure}

The chapter is organized in four parts that respectively treat simplicial homology (\S \ref{Dowker}), the recent and related theory of path homology (\S \ref{Path}), topological data analysis (\S \ref{TDA1D}), and sheaf theory (\S \ref{Sheaves}). Simplicial homology is the most conceptually and computationally ubiquitous algebraic invariant of a reasonably generic space, and we apply it to the analysis of computer code by considering special simplicial complexes that encode relations between program assignments and variables, and that do not even have to be explicitly formed in order to obtain useful invariants. Meanwhile, path homology is an important and quite new theory that defines high-dimensional topological invariants of directed graphs and as such is very promising for cyber-oriented applications such as the analysis of control flow. Our treatment of topological data analysis is very brief and restricted to one dimension, where it is possible to introduce and exploit the morals of topological persistence to the useful end of statistical mixture estimation without invoking the algebraic machinery of persistent homology. Finally, our treatment of sheaves is largely self-contained and developed in service of detecting critical nodes in wireless networks. 

Throughout this chapter, our focus on intrinsically discrete structures, realistic applications, and space constraints entail a somewhat idiosyncratic treatment. For example, the word ``functor'' and its variants do not occur outside this sentence, though we point out the functoriality of simplicial homology without invoking the formalism of category theory.

\section{Dowker homology to analyze complexity of source and binary code}\label{Dowker}

In this section, we introduce a class of data structures called abstract simplicial complexes that model interactions of arbitrary order, generalizing graphs, which model interactions of order two. We illustrate how these data structures can model well-behaved shapes and compute the basic topological invariant of homology by transporting these structures into the realm of linear algebra. Finally, we demonstrate how these ideas can characterize source and binary code. The same ideas could be applied to bipartite structures such as interactions between processes and files, clients and servers, etc. 

\subsection{\label{sec:Simplicial}Simplicial complexes and their homology}

Although topology is generally thought of as the study of spaces under continuous transformations, its intellectual roots are in combinatorial models of spaces.  While these combinatorial models are typically discarded once the theory is developed, they are ideally suited for describing cyber applications.  Abstract simplicial complexes are among the easiest of these combinatorial models to define and apply.

\begin{definition}
An \emph{abstract simplicial complex} \index{abstract simplicial complex}\index{simplicial!complex} is a family $\Delta$ of finite subsets (called \emph{simplices}) of a set $V = \{v_0, \dots, v_p\}$ of \emph{vertices} such that if $X \in \Delta$ and $\emptyset \ne Y \subseteq X$, then $Y \in \Delta$.
\footnote{
In other words, an abstract simplicial complex is a hypergraph with all sub-hyperedges. 
}  Usually, we write simplices with square brackets $[v_0, \dots, v_p]$.  The \emph{dimension} of a simplex $[v_0, \dots, v_p]$ is $p$, which is one less than its cardinality as a set.  A simplex that is the subset of no other simplex is called a \emph{facet}\index{facet}.
\end{definition}

When describing the local structure of a simplicial complex, it is often useful to delineate which simplices are subsets of each other.  If $a$ and $b$ are simplices of a simplicial complex $X$ and $a \subseteq b$, we say that ``$a$ is a \emph{face}\index{face} of $b$'' or equivalently that ``$b$ is a \emph{coface}\index{coface} of $a$.''  These relationships determine the topology of an abstract simplicial complex, in terms of its open and closed subsets.  A \emph{closed set}\index{closed set} $A$ of a simplicial complex contains every possible subset of every element of $A$.  The \emph{star}\index{star} of a subset $A$ of a simplicial complex consists of the set of all simplices containing an element of $A$.  An \emph{open set} of an abstract simplicial complex is one that can be written as a union of stars.

For example, let $\Delta$ be given by all nonempty subsets of sets in $\{ \{1,2\},\{1,3\},\{2,3,4\},\{5\}\}$. Then $\Delta$ is an abstract simplicial complex of dimension $2 = |\{2,3,4\}|-1$; Figure \ref{fig:simpleASC} shows a geometric realization of $\Delta$.

\begin{figure}[htb]
\centering
\includegraphics[trim = 70mm 120mm 70mm 120mm, clip, keepaspectratio, width=.5\textwidth]{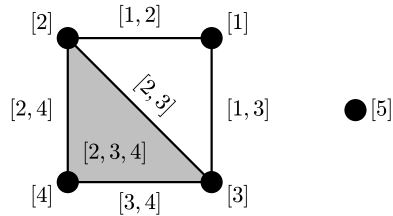}
\caption[Geometric realization of an abstract simplicial complex]{A geometric realization of the abstract simplicial complex $\Delta = (2^{\{1,2\}} \cup 2^{\{1,3\}} \cup 2^{\{2,3,4\}} \cup 2^{\{5\}}) \backslash \emptyset \subseteq 2^V$ with $V = {\{1,2,3,4,5\}}$. The expression of an abstract simplicial complex as a nondegenerate union of power sets manifestly reflects its facets.}
\label{fig:simpleASC}
\end{figure}


In the abstract simplicial complex $\Delta$, the set $A= \{[2,4],[2],[4]\}$ is a closed set but $B=\{[1,2],[1,3],[1],[5]\}$ is not, because $[3]$ is a face of $[1,3]$ that is not contained in $B$.  On the other hand, $B$ is the union of the star over $[1]$ and the star over $[5]$, so $B$ is an open set.  A set can be both open and closed; $\{[5]\}$ is such as set.

Functions that preserve the simplices of abstract simplicial complexes are afforded special status, and are called \emph{simplicial maps}.  These help characterize salient features of abstract simplicial complexes.

\begin{definition}
A \emph{simplicial map}\index{simplicial!map} $f : \Delta \to \Gamma$ from one abstract simplicial complex $\Delta$ to another $\Gamma$ is a function on vertices such that each simplex $\sigma = [v_0, \dots, v_p]$ of $\Delta$ is taken to a simplex $f(\sigma) = [f(v_0), \dots, f(v_p)]$ of $\Gamma$.  
\end{definition}
In the image $f(\sigma)$, repeated vertices count as one vertex.  This means that simplicial maps may decrease the dimension of a simplex but not increase it.

Simplicial maps immediately give rise to the notion of isomorphic abstract simplicial complexes\index{simplicial!isomorphism}: $\Delta$ and $\Gamma$ are \emph{isomorphic} if there are simplicial maps $f: \Delta \to \Gamma$ and $g: \Gamma \to \Delta$ such that $f = g^{-1}$ and $g = f^{-1}$.  Isomorphisms are a natural equivalence relation on abstract simplicial complexes, and generalize the idea of relabeling vertices in a simplicial complex.

It is rather computationally difficult to study abstract simplicial complexes and simplicial maps directly.  It is much easier to work by analogy: transform abstract simplicial complexes into vector spaces, and simplicial maps into linear maps.  The way we will do this is by way of a construction called \emph{simplicial homology}.  The construction is a two-step process, in which we first transform each abstract simplicial complex into an algebraic construction called a \emph{chain complex} and each simplicial map transforms into a \emph{chain map}.  From there, chain complexes and chain maps allow us to compute topological invariants via linear algebra.


\begin{figure}[htb]
\begin{center}
\includegraphics[trim = 60mm 110mm 60mm 115mm, clip, keepaspectratio, width=.7\textwidth]{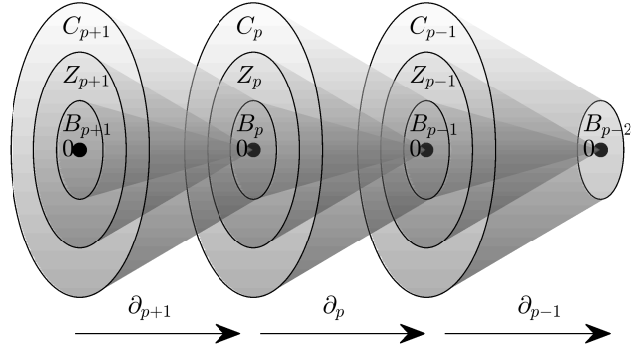}
\caption[Schematic of a chain complex]{Schematic picture of a chain complex.}
\label{fig:ChainComplex}
\end{center}
\end{figure}

\begin{definition}\label{1def:ChainComplex}
A \emph{chain complex} \index{chain complex} over a field $\mathbb{F}$ is a pair of sequences (indexed by $p \in \mathbb{N}$ or $\mathbb{Z}$ depending on context) of $\mathbb{F}$-vector spaces $C_\bullet = \{\dotsc,C_{p-1},C_p,\dotsc\}$ and linear \emph{boundary operators} \index{boundary operator} $\partial_p : C_p \rightarrow C_{p-1}$ such that $\partial_{p-1} \circ \partial_p \equiv 0$. This can be schematically depicted as in Fig. \ref{fig:ChainComplex}, and written as (for $p \in \mathbb{N}$)
\begin{equation}
\label{eq:chainComplex}
\xymatrix{
  \dotsb \ar[r] & C_{p+1} \ar[r]^{\partial_{p+1}} & C_p \ar[r]^{\partial_p} & C_{p-1} \ar[r]^{\partial_{p-1}} & \dotsb \ar[r]^{\partial_1} & C_0 \ar[r]^{\partial_0} & 0.
}
\end{equation}
\end{definition}

Given an abstract simplicial complex $\Delta$, let $C_p(\Delta)$ be the $\mathbb{F}$-vector space generated by basis elements $e_{(v_0,\dots,v_p)}$ corresponding to \emph{oriented simplices} of \emph{dimension} $p$ in $\Delta$. This essentially means that if $\sigma$ is a permutation acting on $(v_0,\dots,v_p)$, then $e_{(v_0,\dots,v_p)} = (-1)^\sigma e_{(v_{\sigma(0)},\dots,v_{\sigma(p)})}$, where $(-1)^\sigma$ indicates the sign of the permutatation $\sigma$.
\footnote{
Thus for example $e_{(v_0,v_1,v_2)} = -e_{(v_0,v_2,v_1)}$.
}
\footnote{
Note that an order on $V$ induces orders (and hence orientations) on simplices in $\Delta$.
}

The simplicial boundary operator $\partial_p$ is now defined to be the linear map acting on basis elements as
\begin{equation}
\label{eq:simplicial}
\partial_p e_{( v_0,\dots,v_p )} := \sum_{j=0}^p (-1)^j e_{ \nabla_j ( v_0,\dots,v_p )}
\end{equation}
where $\nabla_j$ deletes the $j$th entry of a tuple. It turns out that this construction yields a \emph{bona fide} chain complex, called the \emph{simplicial chain complex}\index{chain complex!simplicial} for $C_\bullet(\Delta)$.  To compute $\partial_{p-1} \circ \partial_p$, we delete two entries from $e_{(v_0,\dots,v_p)}$.  There are two different ways we can do this: first $i$ and then $j$, or first $j$ and then $i$.  These two ways yield opposite signs, which cancel all of the terms in the sum.

Like the structure-preserving nature of simplicial maps for abstract simplicial complexes, there are structure preserving \emph{chain maps}\index{chain map} for chain complexes.  They are defined by way of diagrams
\begin{equation*}
\xymatrix{
\dots \ar[r] & C_{p+1} \ar[r]^{\partial_{p+1}} \ar[d]^{m_{p+1}} & C_{p} \ar[r]^{\partial_{p}} \ar[d]^{m_{p}} & C_{p-1} \ar[r]^{\partial_{p-1}} \ar[d]^{m_{p-1}} & \dotsb\\
\dots \ar[r] & C'_{p+1} \ar[r]_{\partial'_{p+1}} & C'_{p} \ar[r]_{\partial'_{p}} & C'_{p-1} \ar[r]_{\partial'_{p-1}}  & \dotsb
}
\end{equation*}
in which composition of consecutive maps is path-independent.  Because in such a diagram
\begin{equation*}
m_{p-1} \circ \partial_p = \partial'_p \circ m_p,
\end{equation*}
it is said to \emph{commute}.

A somewhat involved but straightforward calculation establishes the following key result about the simplicial chain complex.
  
\begin{proposition}
 Every simplicial map $f: \Delta \to \Gamma$ between abstract simplicial complexes induces a chain map $f_\bullet : C_\bullet(\Delta) \to C_\bullet(\Gamma)$ between their simplicial chain complexes.
\end{proposition}

While chain complexes distill abstract simplicial complexes into the realm of algebra, they are still rather complicated.  Moreover, the simplicial chain complex contains combinatorial, non-topological information.  \emph{Homology} is a convenient, linear algebraic summary for a chain complex that still preserves the structure of chain maps.  Additionally, the homology of the simplicial chain complex is a topological invariant.

\begin{definition}
Writing $Z_p := \text{ker } \partial_p$ and $B_p := \text{im } \partial_{p+1}$, the \emph{homology} \index{homology}\footnote{Homology is readily defined over rings, with the integers $\mathbb{Z}$ serving as the case through which all others factor via the universal coefficient theorem (which, incidentally, gave rise to the topics of category theory and homological algebra). However, most practical considerations require fields, and so we restrict the definition above accordingly.} of \eqref{eq:chainComplex} is the sequence of quotient spaces
\begin{equation}
H_p := Z_p/B_p.
\end{equation}
The \emph{Betti numbers} \index{Betti number} are $\beta_p := \text{dim } H_p = \text{dim } Z_p - \text{dim } B_p$.
\hfill{$\Box$}
\end{definition}

The essential point of this construction is that homology transforms chain complexes into vector spaces and chain maps into linear maps.

\begin{proposition}
  Every chain map $m_\bullet : C_\bullet \to D_\bullet$ induces a family of linear maps $(m_*)_p : H_p(C_\bullet) \to H_p(D_\bullet)$ between homology spaces, one for each $p$.
\end{proposition}

As an immediate consequence, a simplicial map $f: \Delta \to \Gamma$ induces a family of linear maps $H_p(C_\bullet(\Delta)) \to H_p(C_\bullet(\Gamma))$ between the homologies of the corresponding simplicial chain complexes.  We will call
\begin{equation*}
  H_p(\Delta) := H_p(C_\bullet(\Delta))
\end{equation*}
the \emph{$p$-simplicial homology} \index{simplicial!homology} of the abstract simplicial complex $\Delta$.  What this means is that if two simplicial complexes are isomorphic, then their simplicial homologies will also be isomorphic vector spaces for every index.  Conversely, if two abstract simplicial complexes have different simplicial homologies, we know that they cannot be isomorphic as simplicial complexes.

For our purposes here, simplicial homology is practically valuable because it underlies \emph{cyclomatic complexity} \index{cyclomatic complexity} \cite{McCabe1976}, which is essentially the first (and only nontrivial) Betti number of a control flow graph treated as an abstract simplicial complex (i.e., edges correspond to $1$-simplices as in Figure \ref{fig:Konigsberg}). Cyclomatic complexity is an archetypal and widely used \cite{Ebert2016} software metric that can guide fuzzing \cite{Duran2011, Iozzo2010} and identification of fault-prone or vulnerable code \cite{Alves2016, Du2019, Medeiros2017}. In \S \ref{Path}, we briefly discuss \emph{path homology}\index{path homology}, which has promise for generalizing cyclomatic complexity to higher dimensions.

\begin{definition}\label{1def:EulerCharacteristic}{\rm
In the event that all but finitely many $\beta_p$ are zero, the \emph{Euler characteristic} \index{Euler characteristic} $\chi := \sum_p (-1)^p \beta_p$ is well-defined. For abstract simplicial complexes, we get the familiar formula $\chi = V - E + F - \dots$, where the terms on the right hand side respectively indicate the numbers of vertices/0-simplices, edges/1-simplices, faces/2-simplices, etc.
}\end{definition}

Moreover, the simplicial Betti numbers $\beta_p$ count the number of voids of dimension $p$ in a geometric realization of an abstract simplicial complex.
\footnote{
Here, 0-dimensional voids amount to connected components. 
}

For example, consider $\Delta = 2^{\{1,2,3\}} \backslash \emptyset$: $C_2(\Delta) = \langle e_{(1,2,3)} \rangle$, where $\langle \cdot \rangle$ indicates the vector span (say, over $\mathbb{R}$); $C_1(\Delta) = \langle e_{(1,2)}, e_{(1,3)}, e_{(2,3)} \rangle$; $C_0(\Delta) = \langle e_{(1)}, e_{(2)}, e_{(3)} \rangle$, and all other $C_p(\Delta)$ are $0$. Using lexicographic indexing for basis elements, we have the matrix representations
\begin{equation}
\label{eq:SimpleComplexMaps}
\partial_2 = \begin{pmatrix} 
1 \\
-1 \\
1 
\end{pmatrix}; \quad
\partial_1 = \begin{pmatrix} 
-1 & -1 & 0 \\
1 & 0 & -1 \\
0 & 1 & 1
\end{pmatrix},
\end{equation}
and all other boundary operators are zero. For example, the boundary of the 2-simplex or ``triangle'' is
\begin{equation}
\partial_2 e_{(1,2,3)} = e_{(1,2)} - e_{(1,3)} + e_{(2,3)}, \nonumber
\end{equation}
or in matrix form 
\begin{equation}
\begin{pmatrix} 
1 \\
-1 \\
1 
\end{pmatrix} 
\begin{pmatrix} 
1 \end{pmatrix} = 
\begin{pmatrix} 
1 \\
-1 \\
1 
\end{pmatrix}.
\end{equation}
Its boundary in turn is
\begin{equation}
\partial_1 \left ( e_{(1,2)} - e_{(1,3)} + e_{(2,3)} \right ) = 0, \nonumber
\end{equation}
or in matrix form 
\begin{equation}
\begin{pmatrix} 
-1 & -1 & 0 \\
1 & 0 & -1 \\
0 & 1 & 1
\end{pmatrix} 
\begin{pmatrix} 
1 \\
-1 \\
1 
\end{pmatrix} = 
\begin{pmatrix} 
0 \\
0 \\
0 
\end{pmatrix}.
\end{equation}
Thus the homology of the boundary of a triangle has $\beta_p = \delta_{p1}$: there is a single void in dimension 1, and none in other dimensions.

As a slightly more detailed example, take $V = \{1,\dots,5\}$ and $\Delta$ to be all nonempty subsets of sets in $\{ \{1,2\},\{1,3\},\{2,3,4\},\{5\}\}$, as in Figure \ref{fig:simpleASC}. We have the chain complex (over $\mathbb{R}$)
\begin{equation}
\label{eq:ExampleComplex}
0 \overset{\partial_3}{\longrightarrow} C_2 \overset{\partial_2}{\longrightarrow} C_1 \overset{\partial_1}{\longrightarrow} C_0 \overset{\partial_0}{\longrightarrow} 0
\end{equation}
where $C_2 = \left \langle e_{(2,3,4)} \right \rangle$, $C_1 = \left \langle e_{(1,2)}, e_{(1,3)}, e_{(2,3)}, e_{(2,4)}, e_{(3,4)} \right \rangle$,  $C_0 = \left \langle e_{(1)}, e_{(2)}, e_{(3)}, e_{(4)}, e_{(5)} \right \rangle$, and the nontrivial boundary operators are (again, lexicographically ordering basis elements)
\begin{equation}
\label{eq:ExampleBoundaries}
\partial_2 = \begin{pmatrix} 0 \\ 0 \\ 1 \\ -1 \\ 1 \end{pmatrix}; \quad \partial_1 = \begin{pmatrix} -1 & -1 & 0 & 0 & 0 \\ 1 & 0 & -1 & -1 & 0 \\ 0 & 1 & 1 & 0 & -1 \\ 0 & 0 & 0 & 1 & 1 \\ 0 & 0 & 0 & 0 & 0
\end{pmatrix}.
\end{equation}
A few row reductions yield that $\text{rank}(\partial_1) = 1$ and $\text{rank}(\partial_2) = 3$, which gives the hard part of Table \ref{table:ExampleComplex}. It follows that $\beta_\bullet = (2,1,0,\dots)$. Indeed, a geometric realization of $\Delta$ has two connected components and one hole.

\begin{table}
    \title{Betti numbers for \eqref{eq:ExampleComplex}} \\
    \begin{tabular}{|c|c|c|c|}
        \hline
        $p$ & $\dim Z_p$ & $\dim B_p$ & $\beta_p$ \\
        \hline
        $0$ & $5$ & $3$ & $2$ \\
        $1$ & $2$ & $1$ & $1$ \\
        $2$ & $0$ & $0$ & $0$ \\
        \hline
    \end{tabular}
    \label{table:ExampleComplex}
\end{table}

As a more intricate example, take $V = \{1,\dots,18\}$ and $\Delta$ to be all nonempty subsets of sets in 
\begin{align}
\{ & \{1,2,3\},\{1,4\},\{5\},\{6,7,8,9\},\{9,10,12\},\{10,11,12\}, \nonumber \\
& \{12,13,16\},\{13,14\},\{13,15\},\{14,15\},\{16,17\},\{17,18\}\}. \nonumber
\end{align}


\begin{figure}[htb]
\centering
\includegraphics[trim = 60mm 120mm 60mm 110mm, clip, keepaspectratio, width=.75\textwidth]{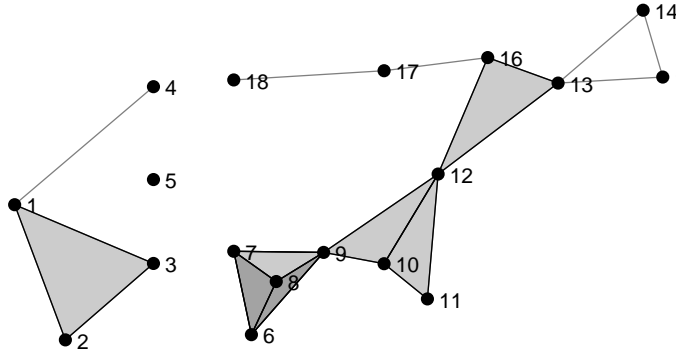}
\caption[ASC]{A geometric realization of an ASC with Betti numbers $\beta_\bullet = (3,1,0,\dots)$.}
\label{fig:ASC}
\end{figure}

Now $\partial_p$ acts on the span of all vectors of the form $e_{(v_0,\dots,v_p)}$ where $\{v_0,\dots,v_p\} \in \Delta$. Meanwhile, a brief calculation shows that $\text{ker } \partial_2 = \text{im } \partial_3 \oplus 0$, $\text{ker } \partial_1 = \text{im } \partial_2 \oplus \left \langle e_{(13,14)} - e_{(13,15)} + e_{(14,15)} \right \rangle$, and $\text{ker } \partial_0 = \text{im } \partial_1 \oplus \left \langle e_{(1)}, e_{(5)}, e_{(6)} \right \rangle = \left \langle e_{(1)}, \dots e_{(18)} \right \rangle$. It follows that $H_p = 0$ for $p \ge 2$, $H_1 = \left \langle e_{(13,14)} - e_{(13,15)} + e_{(14,15)} \right \rangle$, and $H_0 = \left \langle e_{(1)}, e_{(5)}, e_{(6)} \right \rangle$: thus $\beta_\bullet = (3,1,0,\dots)$. A geometric realization of $\Delta$ has three connected components and one hole, as shown in Figure \ref{fig:ASC}.

\subsection{\label{sec:Dowker}Dowker homology}
 
For finite sets $X$, $Y$ and a relation $R \subseteq X \times Y$,
\footnote{
Recall that this inclusion is just the definition of a generic relation between $X$ and $Y$.
}
we can form two abstract simplicial complexes. The first has vertex set $X$ and simplices generated by finite subsets of $R(\cdot,y) := \{x \in X : (x,y) \in R\}$ for $y \in Y$; the second has vertex set $Y$ and simplices generated by finite subsets of $R(x,\cdot)$ for $x \in X$. Remarkably, these two abstract simplicial complexes are topologically equivalent under the very strong notion of homotopy \cite{Dowker1952}, and either is referred to as a \emph{Dowker complex} \index{Dowker complex} of the relation $R$. Almost as remarkably, the $\mathbb{F}_2$ homology of Dowker complexes can be computed directly from the relation $R$ in about 50 lines of straightforward MATLAB\textsuperscript{\textregistered} code, and the only computationally intensive part is computing the rank of the boundary matrices.

For example, consider the relation specified by the 0-1 matrix
\footnote{We adapt this example (which also informs a previous one) from \cite{GhristEAT}.}
\begin{equation}
\label{eq:ExampleRelation}
R = \begin{pmatrix} 
0 & 1 & 0 & 0 & 0 & 1 & 0 \\
0 & 0 & 0 & 1 & 0 & 1 & 1 \\
1 & 1 & 0 & 1 & 0 & 0 & 1 \\
1 & 0 & 0 & 1 & 0 & 0 & 0 \\
0 & 0 & 1 & 0 & 1 & 0 & 0 
\end{pmatrix}.
\end{equation}
Taking the choice of vertex set $X = \{1,\dots,5\}$, we get the same chain complex as \eqref{eq:ExampleComplex} and \eqref{eq:ExampleBoundaries}, but over $\mathbb{F}_2$: i.e., all signs in \eqref{eq:ExampleBoundaries} are ignored. The matrix element $(\partial_p)_{jk}$ indicates whether or not the set corresponding to the $j$th basis element in $C_{p-1}$ is contained in the set corresponding to the $k$th basis element in $C_p$. Computer calculations yield the same results as in Table \ref{table:ExampleComplex}.

Dowker complexes have a long history of applications to social disciplines under the aegis of ``Q-analysis'' \cite{Atkin1974}; however, only recently have applications gained any wider traction, e.g. to navigation and mapping \cite{Ghrist2012}, lower bounds in privacy analyses \cite{Erdmann2017}, and analyses of weighted digraphs \cite{Chowdhury2018}. The preceding example highlights a circle of ideas that is very interesting for cyber applications. In the following, we detail another application (in many ways mirroring \cite{robinson2017sheaf}) of Dowker complexes to the analysis and characterization of ``straight-line'' source code and/or \emph{basic blocks} \index{basic block} in binary code (i.e., sequences of instructions without control flow).

Programs are fairly simple to define and have a simple decision procedure for determining equality: a program is a string in some language, and two programs are equal if and only if they are equal as strings. Meanwhile, functions are also fairly simple to define (even in the context of computers, via the theory of denotational semantics \cite{nielson1992semantics}), though the problem of determining equality of functions within even simple classes is undecidable \cite{richardson1969some}. However, \emph{algorithms} are notoriously hard to define, and though there is a sort of order structure on reasonable definitions \cite{yanofsky2017galois}, all of the definitions that are substantively different from programs also lead to undecidable equality problems (but see, e.g. \cite{Taherkhani2010, Mesnard2016, Shalaby2017} for the sorts of heuristics used in practice).

To illustrate this notion, consider the sets of ``algorithms'' in Fig. \ref{fig:toy1} and \ref{fig:toy2}: each set has the same inputs (\texttt{a}, \texttt{b}, \texttt{c}, and \texttt{d}), and outputs (\texttt{q} and \texttt{x}), albeit computed differently. Absent notions of control flow (e.g., conditional branches or loops), it is easy to define a relation between variables and assignments and construct the corresponding Dowker complex. In these examples, the homology classes associated to ``primitive'' algorithms on the left are preserved under compilation-like rewrites, though additional homology classes can be introduced by ``tearing apart'' high-arity assignments into low-arity ones. More formally and suggestively, the primitive notion of ``decompilation'' indicated here is an injective simplicial map, and thus induces a homomorphism on homology.

\begin{figure}[htb]
\begin{center}
\includegraphics[trim = 40mm 65mm 45mm 85mm, clip, keepaspectratio, width=.21\textwidth]{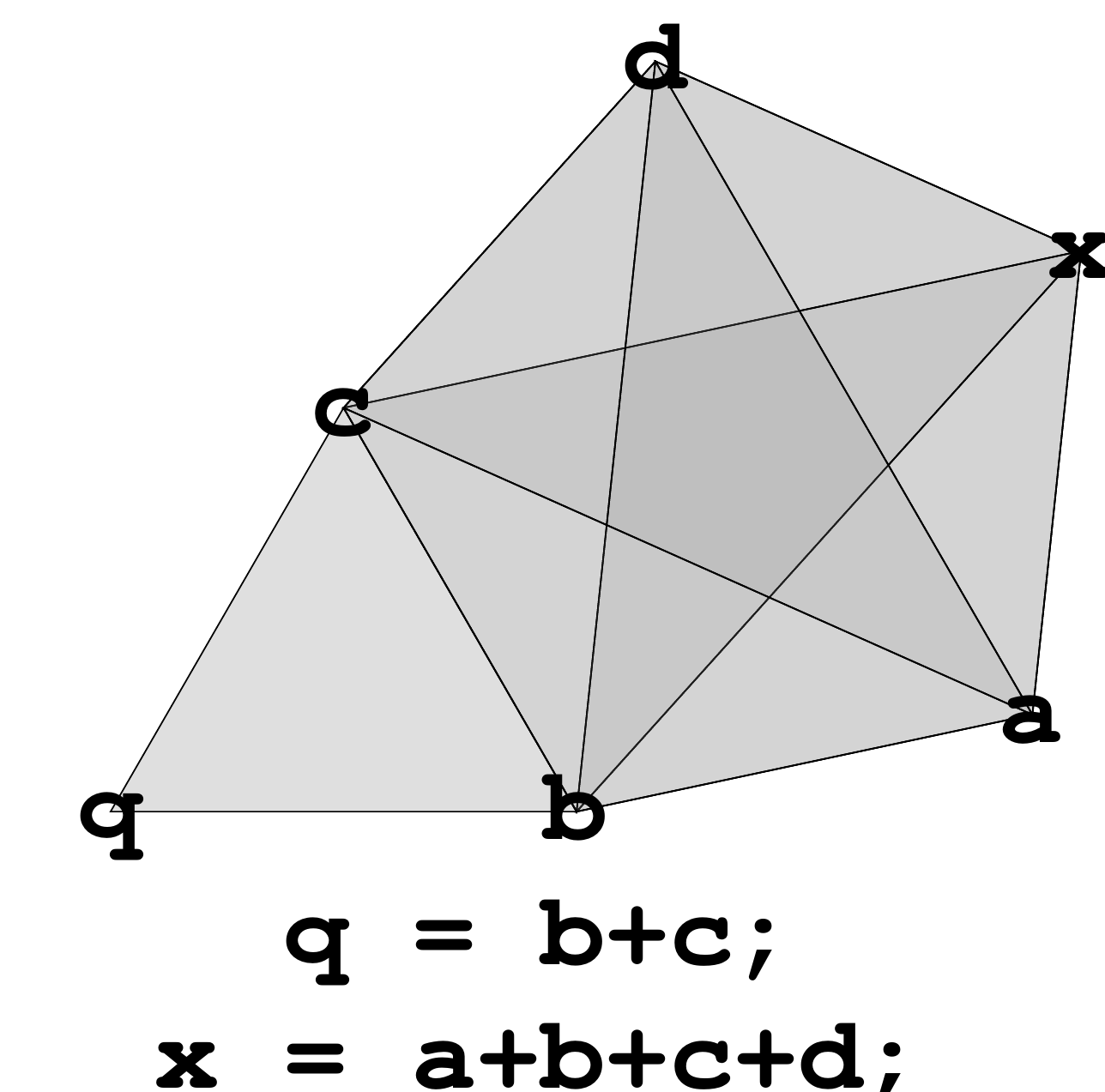} \quad
\includegraphics[trim = 35mm 65mm 45mm 90mm, clip, keepaspectratio, width=.21\textwidth]{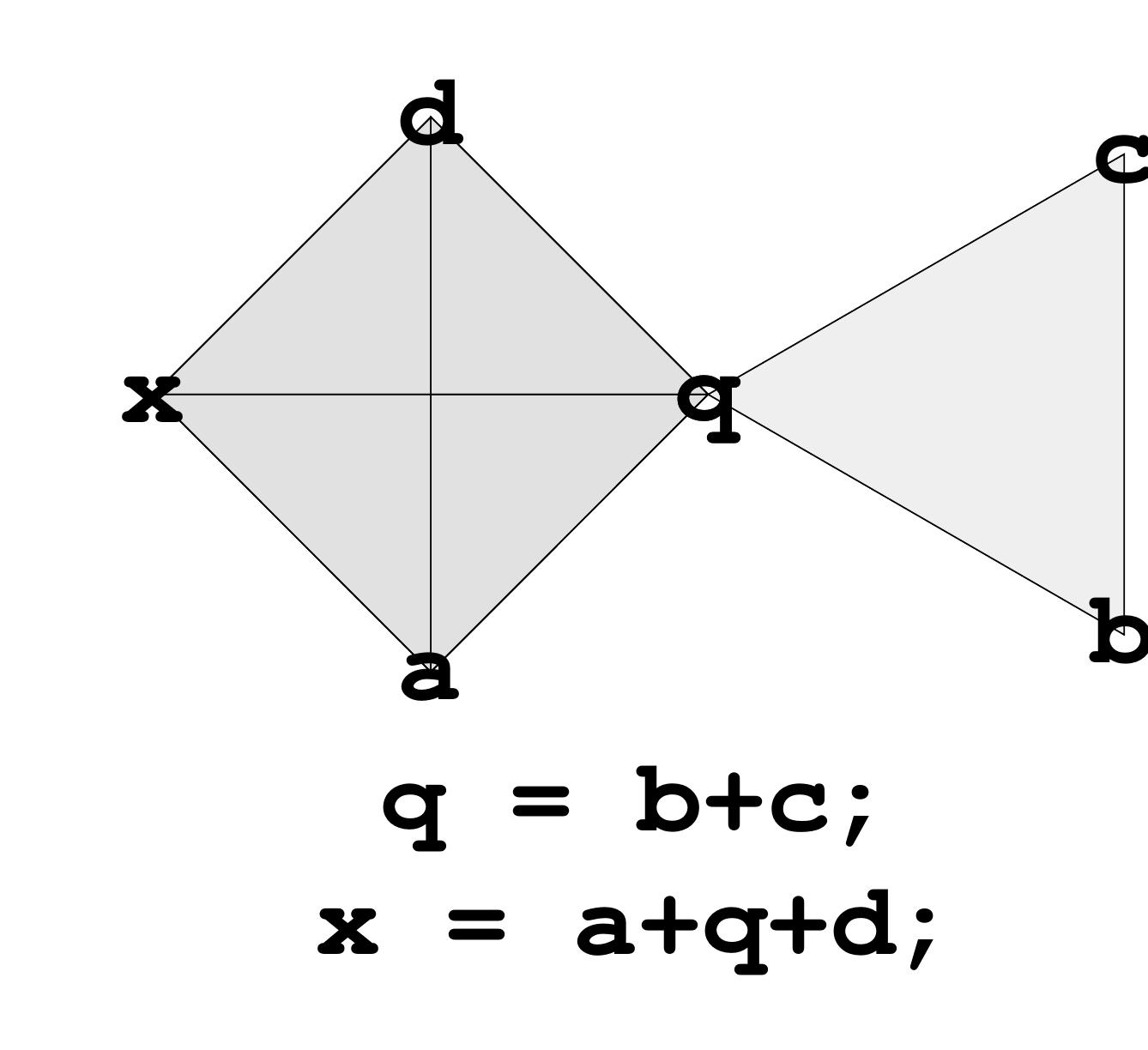} \quad
\includegraphics[trim = 0mm 10mm 75mm 0mm, clip, keepaspectratio, width=.3\textwidth]{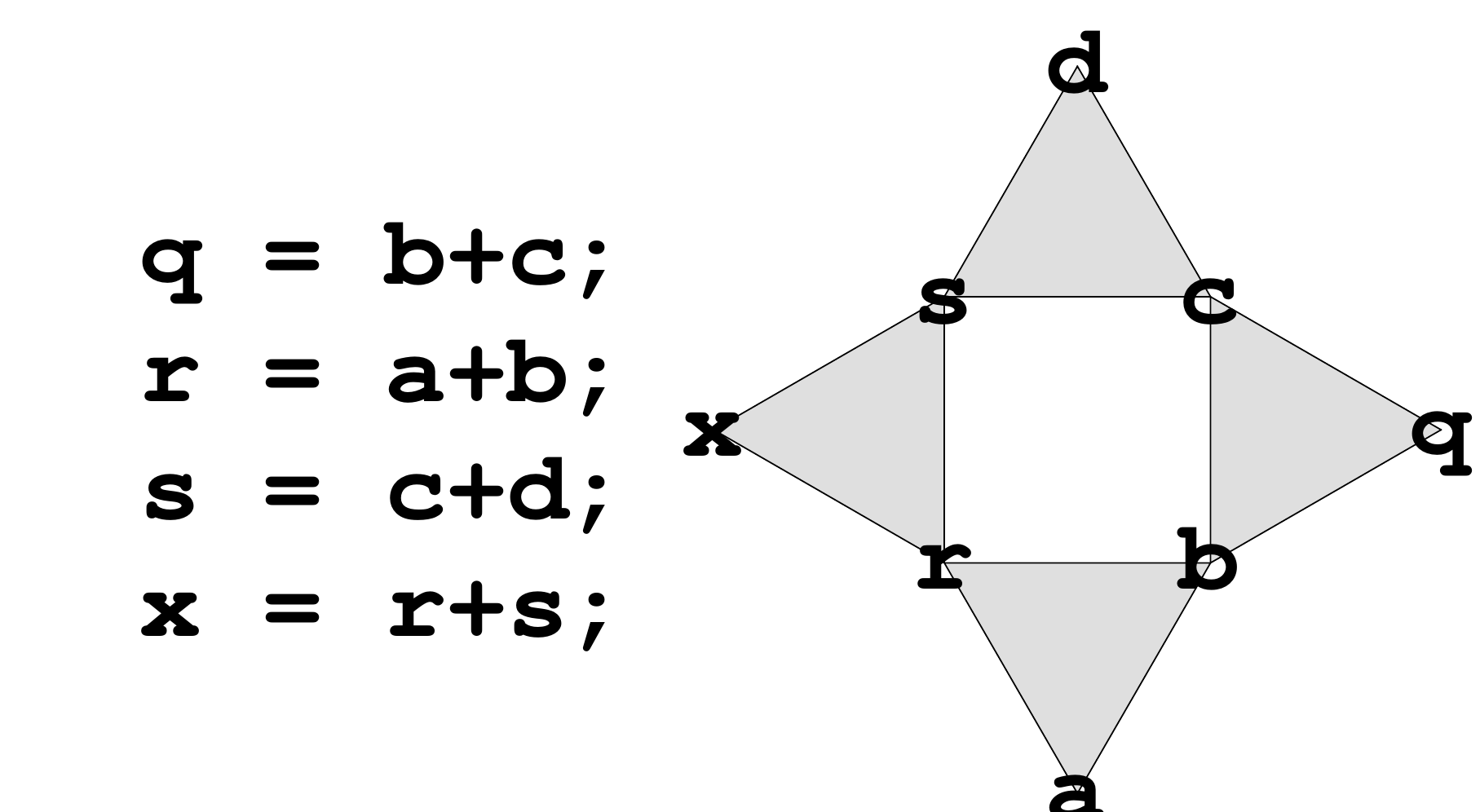}
\end{center}
\caption[Dowker complexes of toy algorithms (1)]{From left to right: Dowker complexes for a toy algorithm, a similar algorithm, and a ``compiled'' version.}
\label{fig:toy1}
\end{figure}

\begin{figure}[htb]
\includegraphics[trim = 35mm 90mm 55mm 90mm, clip, keepaspectratio, width=.32\textwidth]{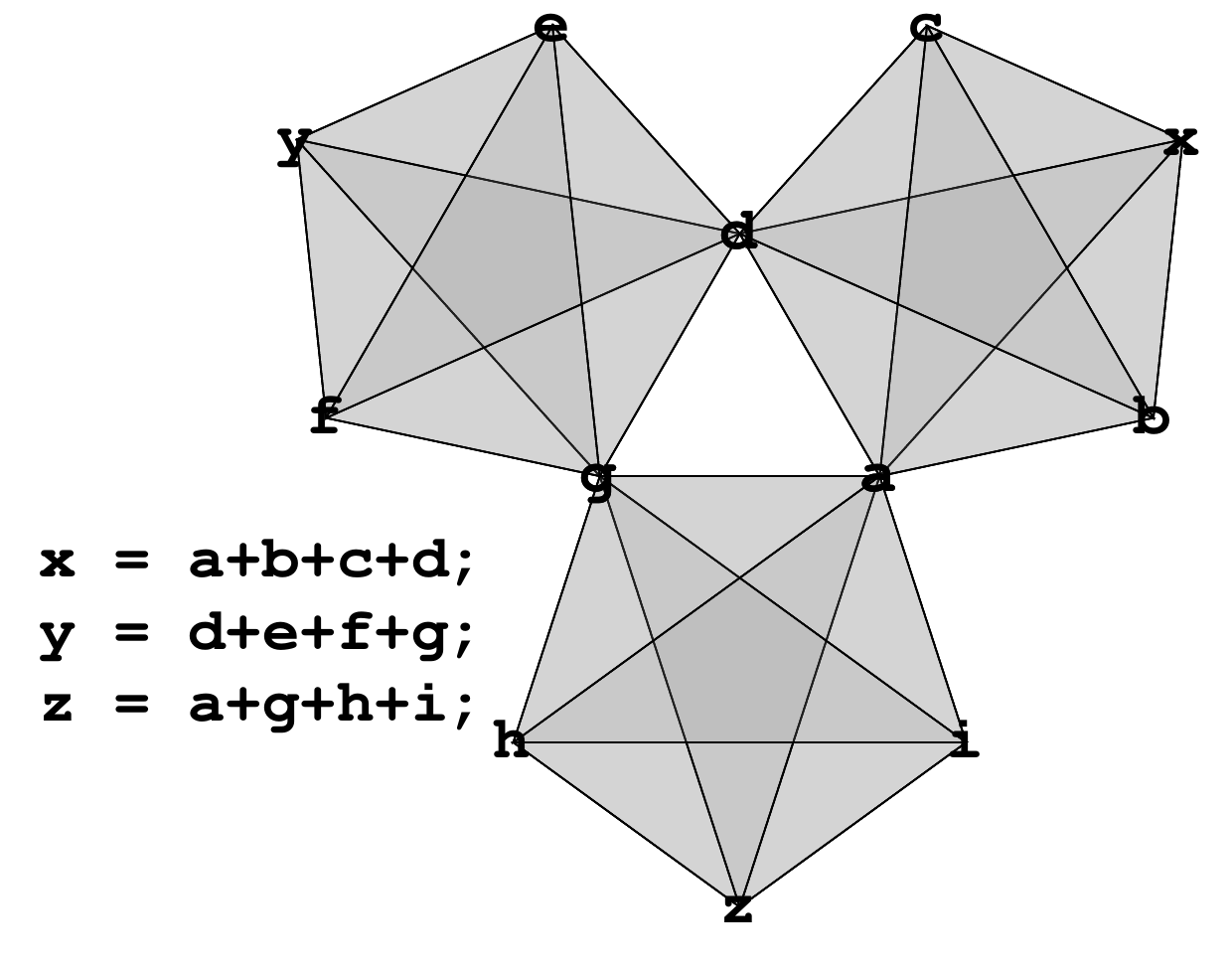}
\includegraphics[trim = 35mm 90mm 55mm 90mm, clip, keepaspectratio, width=.32\textwidth]{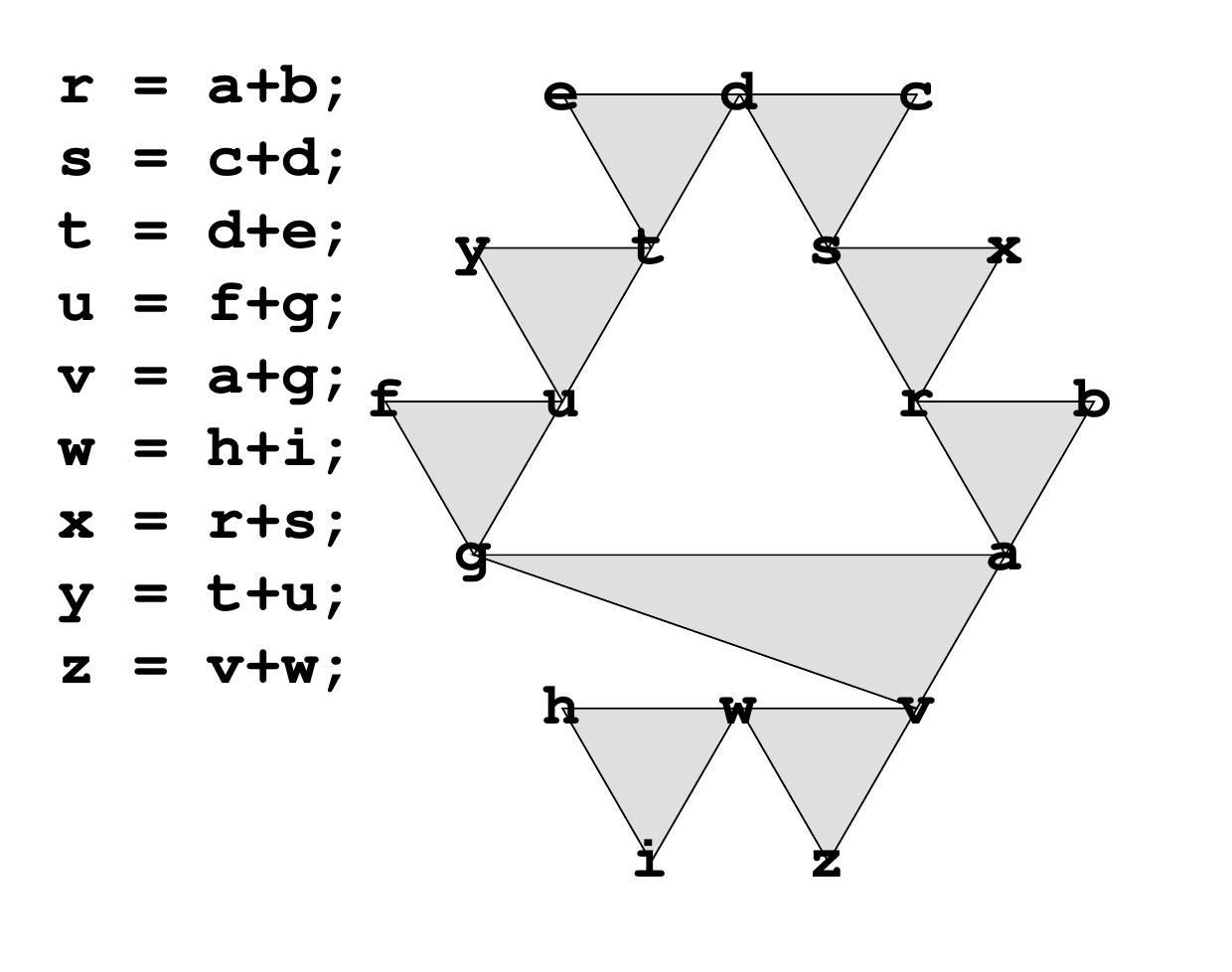}
\includegraphics[trim = 35mm 90mm 55mm 90mm, clip, keepaspectratio, width=.32\textwidth]{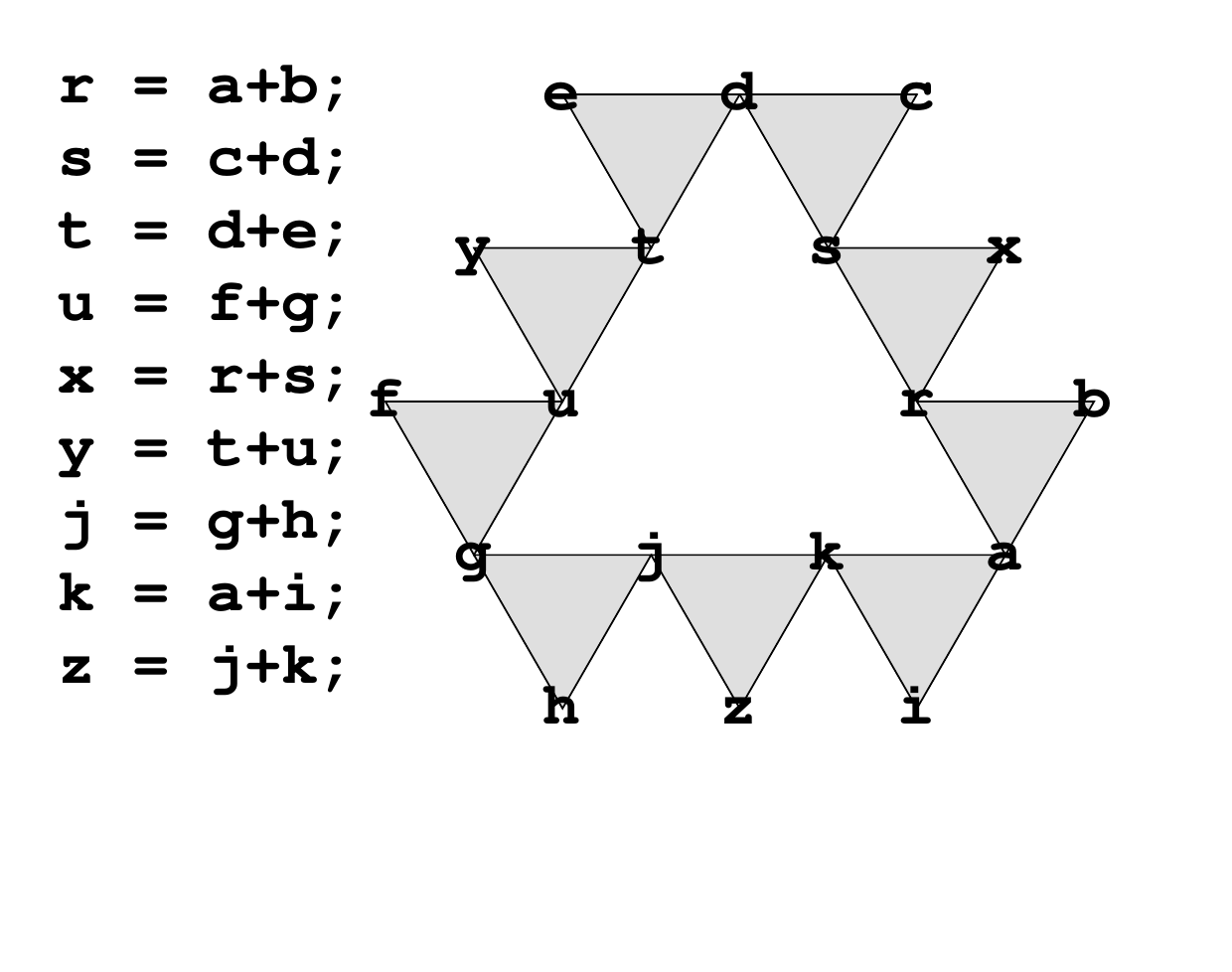}
\caption[Dowker complexes of toy algorithms (2)]{From left to right: Dowker complexes for a toy algorithm, and two ``compiled'' versions.}
\label{fig:toy2}
\end{figure}

Notwithstanding the problems of defining algorithms and the undecidability of gauging equality (much less computing a principled similarity) of algorithms in general, Dowker homology can capture salient information about straight-line or basic block algorithms.
\footnote{To handle control flow nicely, F. R. Genovese (private communication) has suggested considering a so-called \emph{\'etal\'e space} building on the sheaf implied by considering subsets of instructions/assignments. However, given \emph{any} construct capable of dealing with basic control flow in the present context, it should be possible to ``desugar'' more complex language semantics to deal with correspondingly more complex control flow and data structures.}  
In this restricted setting, it is not hard to identify various narrow classes of algorithms that admit reasonable definitions. 

For example, sorting networks are fixed compositions of pairwise compare/swap operations that guarantee to sort an input tuple of a given size \cite{knuth1997art}, and fixed-size matrix multiplication algorithms are essentially rank-1 decompositions of a particular tensor \cite{landsberg2017geometry}. In both cases, the formulation of optimal algorithms is a nontrivial problem. For matrix multiplication, the ``na\"ive'' algorithm was originally improved upon by \cite{strassen1969gaussian,winograd1971multiplication}, which showed how to multiply two $2 \times 2$ matrices with only 7 scalar multiplications (versus 8 for the na\"ive approach). Although these instances are known to be optimal, all that is known for the $3 \times 3$ case is that somewhere between 19 and 23 scalar multiplications are required (versus 27 for the na\"ive approach), and over noncommutative (resp., commutative) rings 23 (resp., 22) scalar multiplications is the best known result, achieved by many inequivalent algorithms \cite{laderman1976noncommutative,johnson1986noncommutative,makarov1986algorithm,courtois2011new} which we analyze below along with the na\"ive algorithms and some ``compiled'' variants where all assignments have two inputs. There is also recent work producing still more $3 \times 3$ algorithms (see, e.g. \cite{chokaev2018two,ballard2019geometry,heule2019new}) and notions of matrix multiplication algorithm equivalence for more general sizes \cite{berger2019equivalent}.

Figures \ref{fig:sorting4} and \ref{fig:sorting56} illustrate how Dowker homology can distinguish between optimal sorting networks: using the negative Euler characteristic as a measure of topological complexity highlights networks that exhibit more comparator reuse and symmetry. 

\begin{figure}[htb]
\includegraphics[trim = 15mm 255mm 120mm 15mm, clip, keepaspectratio, width=.4\textwidth]{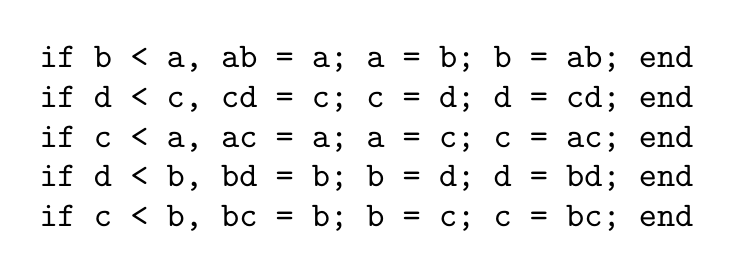}
\includegraphics[trim = 45mm 90mm 45mm 90mm, clip, keepaspectratio, width=.16\textwidth]{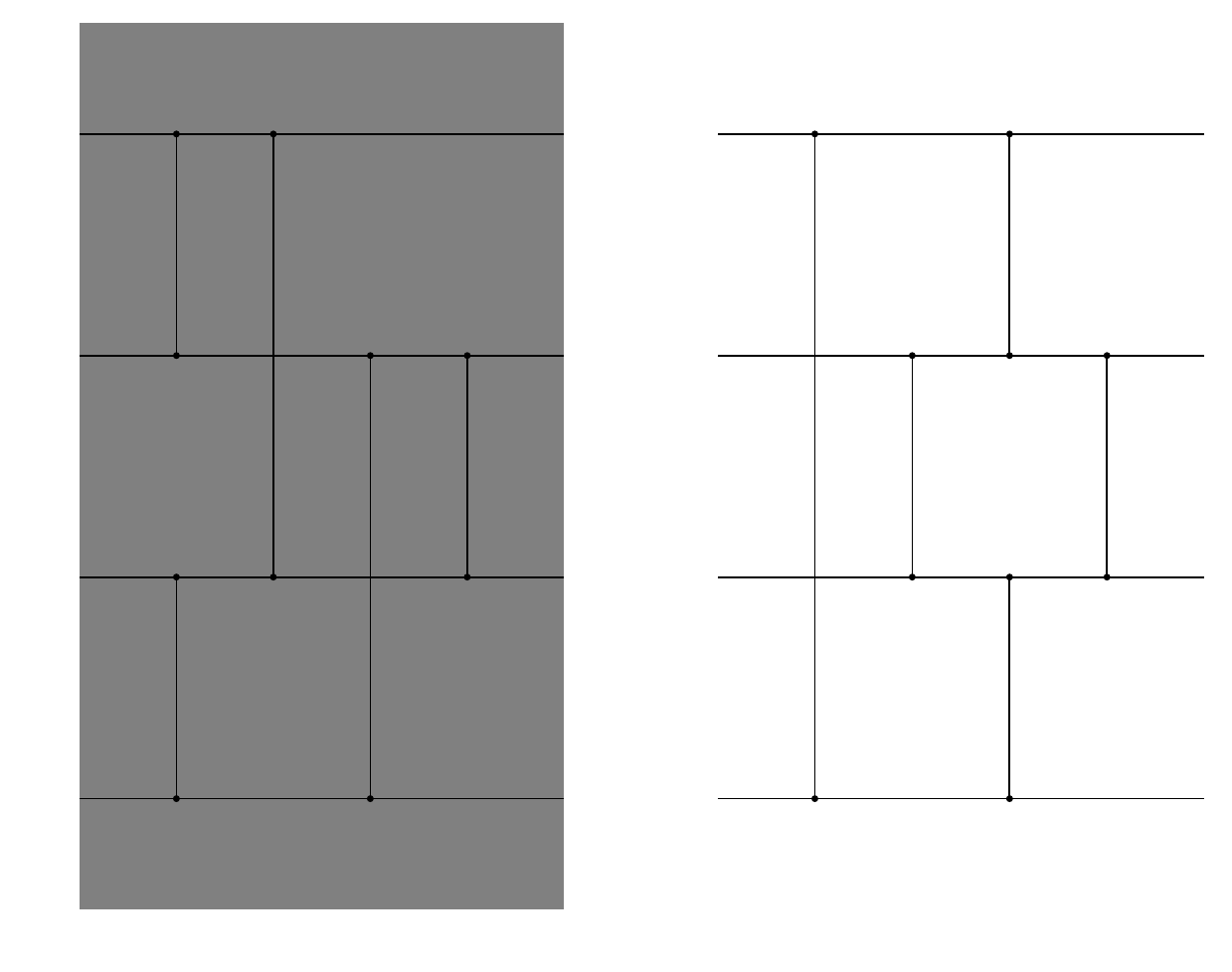}
\includegraphics[trim = 15mm 255mm 120mm 15mm, clip, keepaspectratio, width=.4\textwidth]{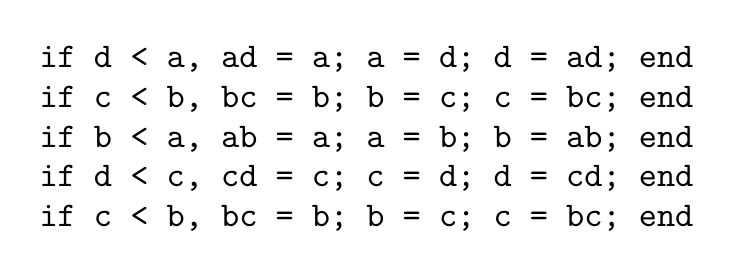}
\caption[Topology of optimal sorting network algorithms on four inputs]{From left to right: code represetation of an optimal sorting network for $n = 4$; graphical representation of the same network (with inputs on left labeled $\texttt{a}$ through $\texttt{d}$ from top down and outputs on right), graphical representation of the other optimal network; code representation of the other optimal network. The graphical representations are shaded by $-\chi$ (lower values are paler) of Dowker complexes formed from code (by treating the statements $\texttt{if k < j}$ as vertices $\texttt{i}_{\texttt{jk}}$.
While the graphical representations are topologically equivalent (specifically, both are homotopic to a figure eight), the Dowker complexes are respectively homotopic to a figure eight and a circle.}
\label{fig:sorting4}
\end{figure}

\begin{figure}[htb]
\includegraphics[trim = 45mm 90mm 45mm 90mm, clip, keepaspectratio, width=.49\textwidth]{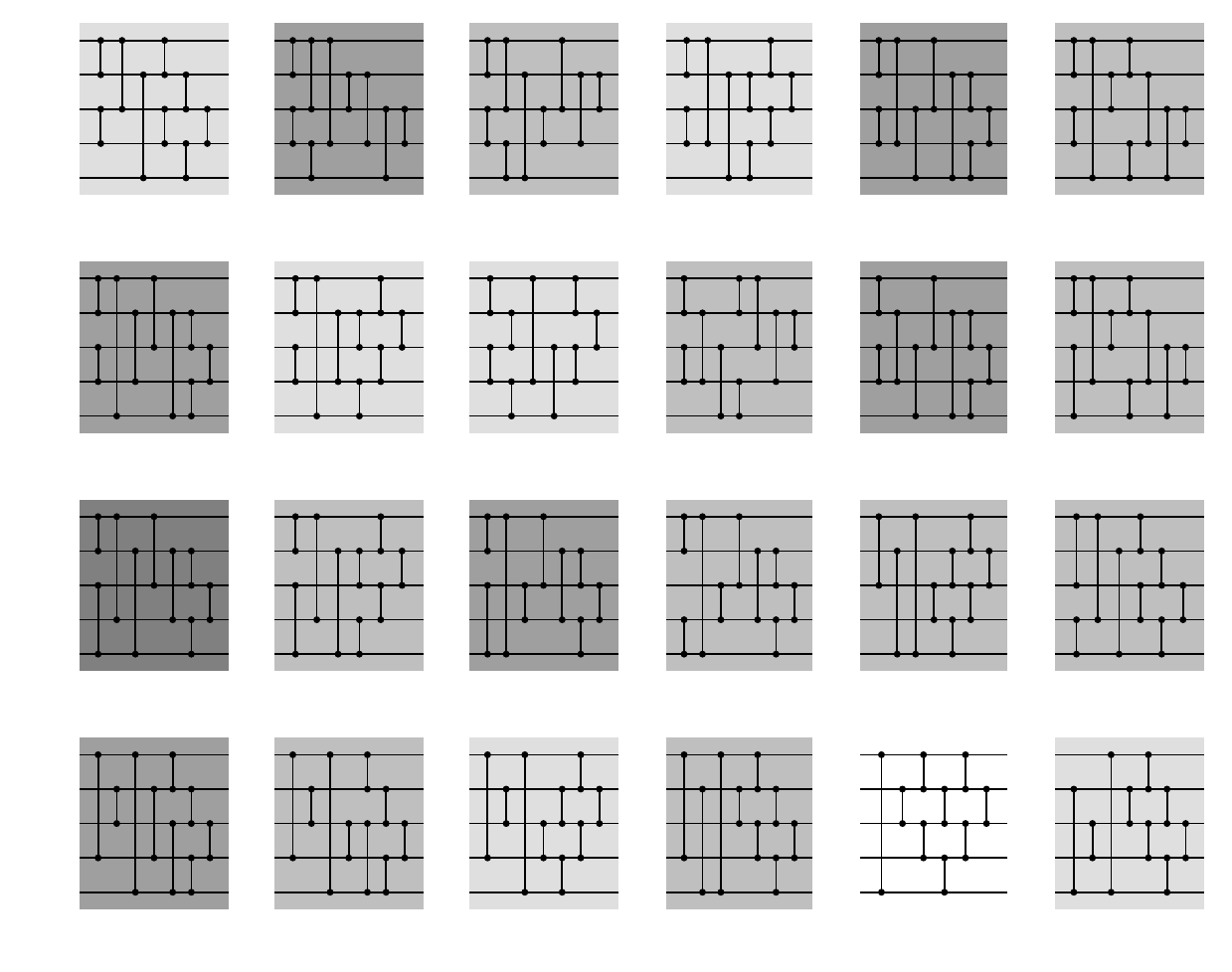}
\includegraphics[trim = 45mm 90mm 45mm 90mm, clip, keepaspectratio, width=.49\textwidth]{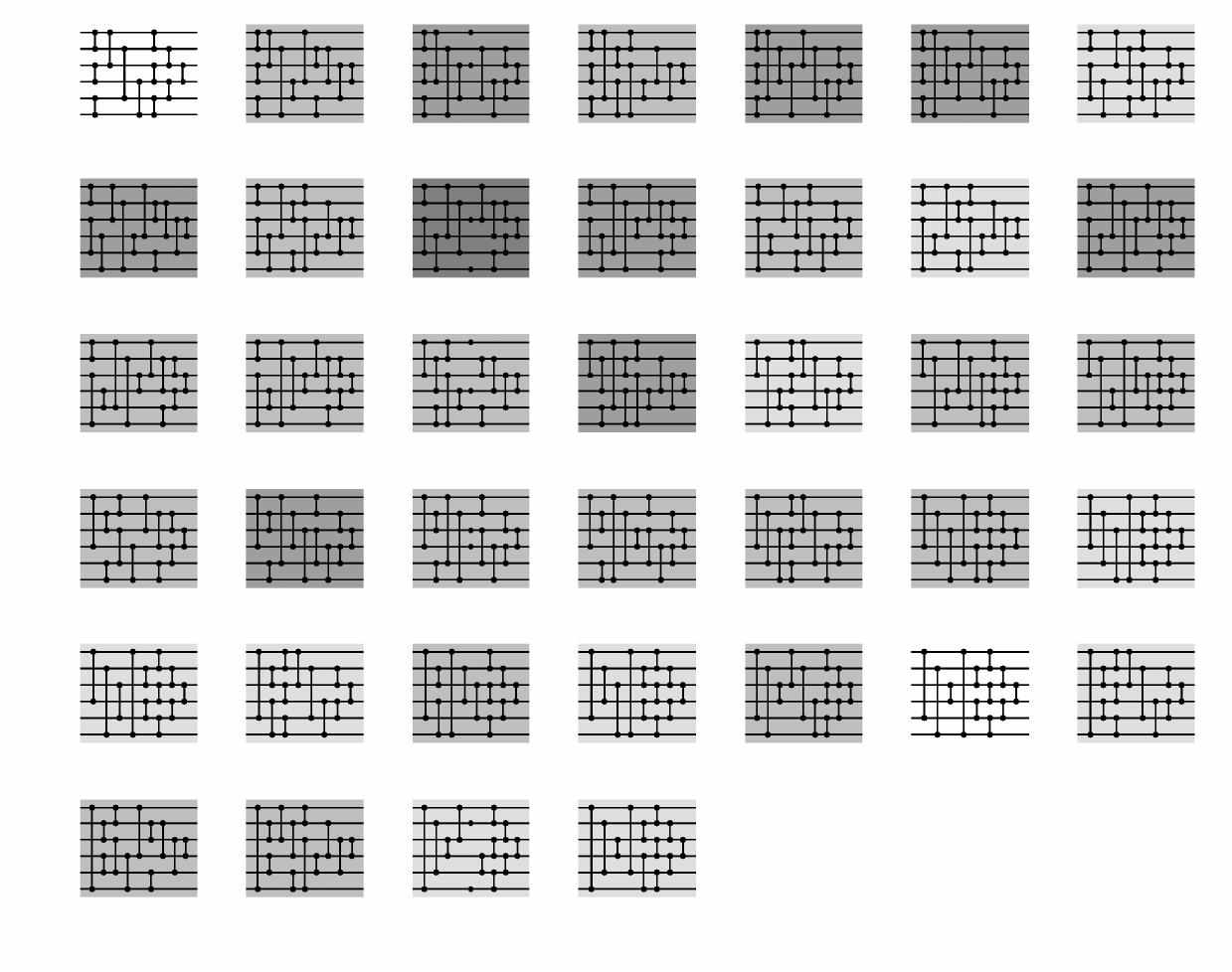}
\caption[Topology of optimal sorting network algorithms on five and six inputs]{Representative sorting networks for $n=5$ (left) and $n=6$ (right) shaded by $-\chi$. Reuse of comparators and symmetry turn out to be signaled by lower (= paler) values.}
\label{fig:sorting56}
\end{figure}


Meanwhile, Figures \ref{fig:mm} and \ref{fig:MatrixMultiplyAlgorithms} give a sense of how matrix multiplication algorithms cluster in meaningful ways when the Betti numbers for Dowker homology are used as features. By computing homologies over local windows of instructions/assignments/line numbers (Figures \ref{fig:naiveThree2}-\ref{fig:Makarov}), we obtain detailed structurally aware features evocative of spectrograms.

\begin{figure}[htb]
\includegraphics[trim = 20mm 105mm 60mm 80mm, clip, keepaspectratio, width=.49\textwidth]{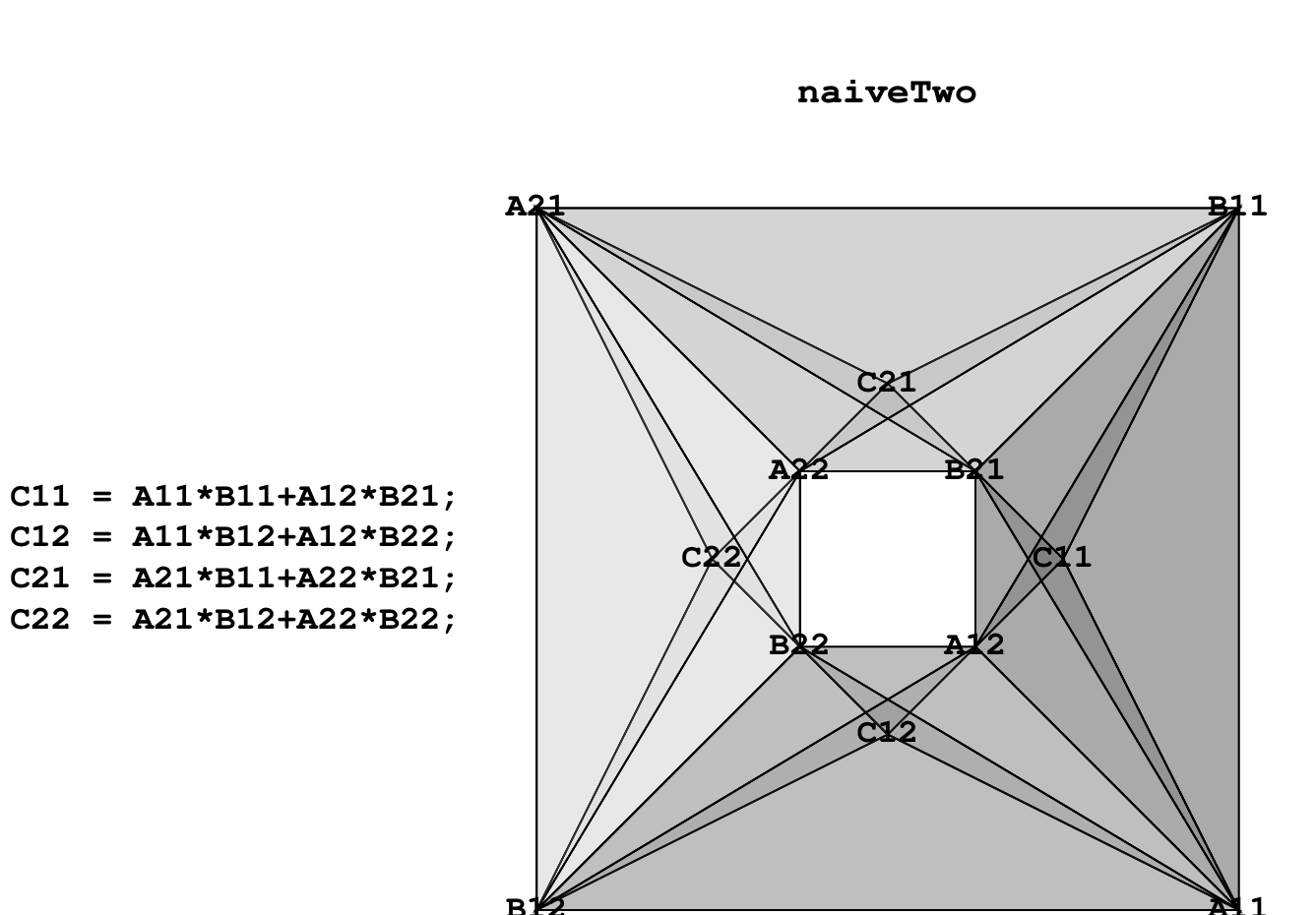}
\includegraphics[trim = 20mm 105mm 60mm 80mm, clip, keepaspectratio, width=.49\textwidth]{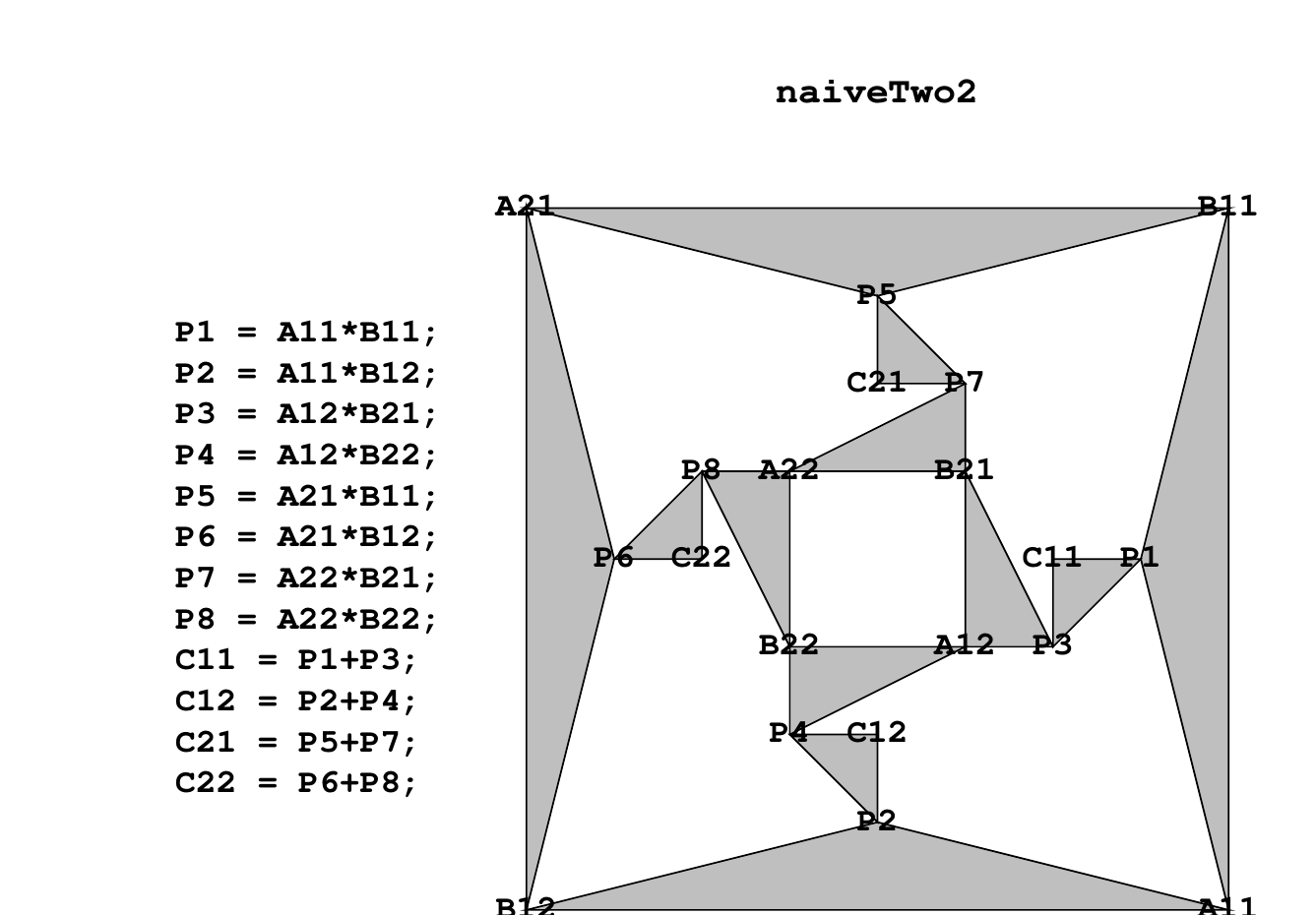}
\caption[Dowker complexes for elementary matrix multiplication algorithms]{From left to right: Dowker complexes for the na\"ive $2 \times 2$ matrix multiplication algorithm (with differently shaded simplices of top dimension), and for a ``compiled'' version.}
\label{fig:mm}
\end{figure}

\begin{figure}[htb]
\includegraphics[trim = 20mm 0mm 20mm 0mm, clip, keepaspectratio, width=\textwidth]{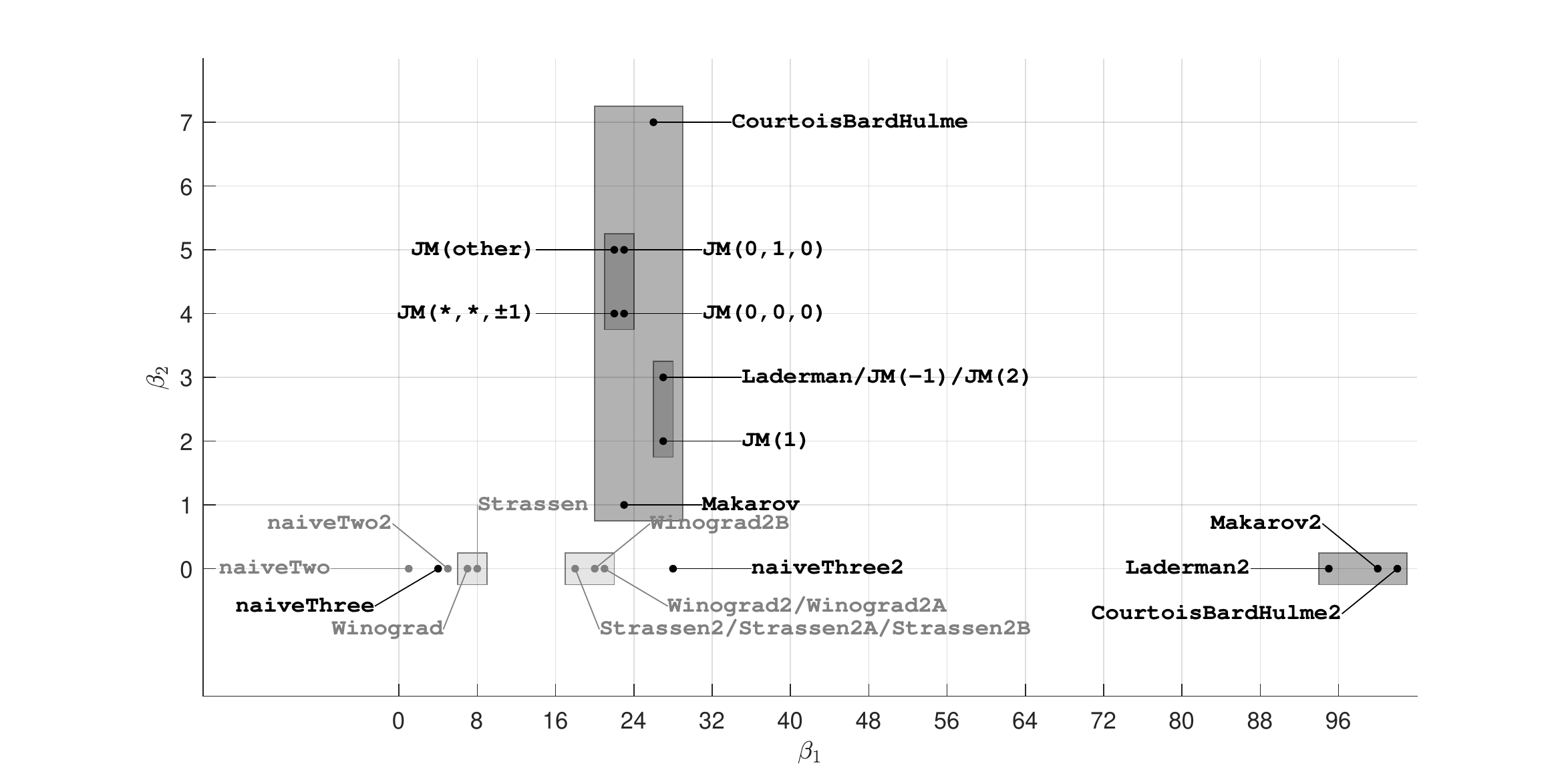}
\caption[Dowker homologies of 48 different matrix multiplication algorithms]{The Betti numbers for Dowker homology of matrix multiplication algorithms are useful features for clustering. There are respectively 18 and 7 inequivalent but similar algorithms from the three-parameter Johnson-McLoughlin family corresponding to the labels \texttt{JM(*,*,$\pm$1)} and \texttt{JM(other)}.}
\label{fig:MatrixMultiplyAlgorithms}
\end{figure}

\begin{figure}[htb]
\includegraphics[trim = 15mm 5mm 50mm 5mm, clip, keepaspectratio, width=\textwidth]{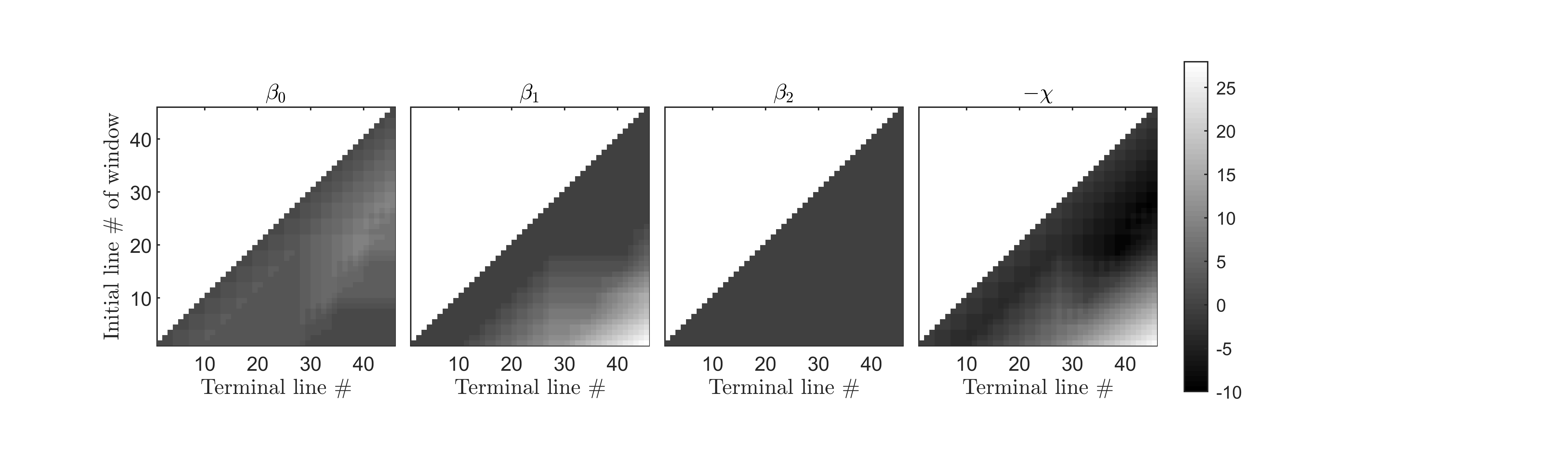}
\caption[Windowed Dowker homology for a ``compiled'' version of na\"ive $3 \times 3$ matrix multiplication]{Windowed Dowker homology for the \texttt{naiveThree2} algorithm, i.e., the ``compiled'' version of na\"ive $3 \times 3$ matrix multiplication. Structural features of lines 1-9, 2-10, \dots, 19-27 vs lines 28-45 are apparent.}
\label{fig:naiveThree2}
\end{figure}

\begin{figure}[htb]
\includegraphics[trim = 15mm 5mm 50mm 5mm, clip, keepaspectratio, width=\textwidth]{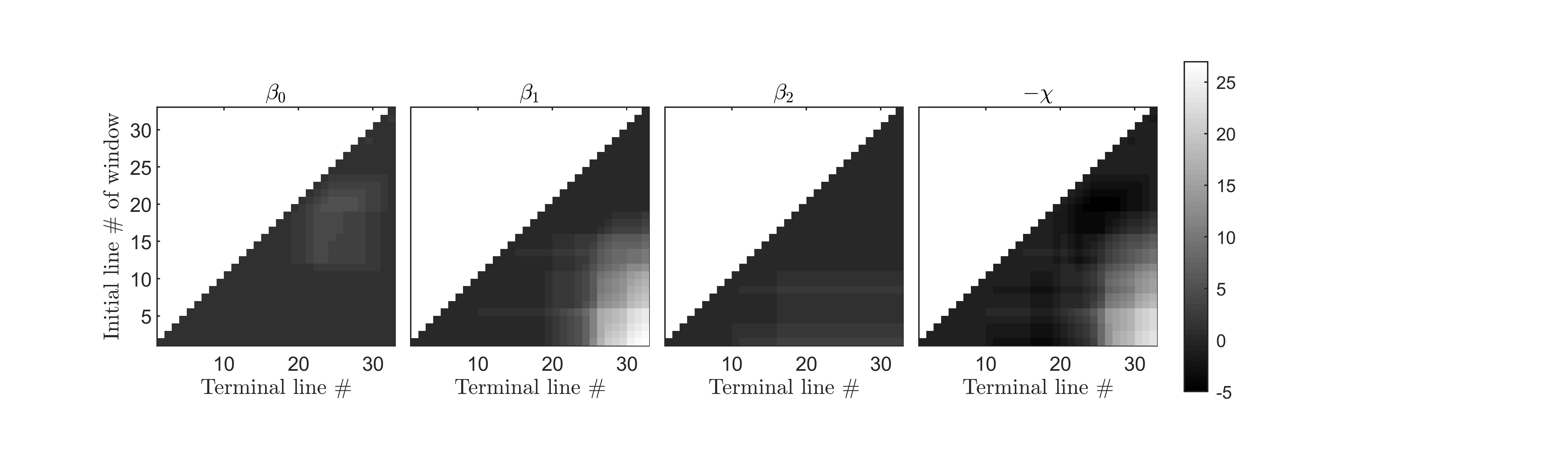}
\caption[Windowed Dowker homology for the \texttt{Laderman} $3 \times 3$ matrix multiplication algorithm]{Windowed Dowker homology for the \texttt{Laderman} $3 \times 3$ matrix multiplication algorithm.}
\label{fig:Laderman}
\end{figure}

\begin{figure}[htb]
\includegraphics[trim = 15mm 5mm 50mm 5mm, clip, keepaspectratio, width=\textwidth]{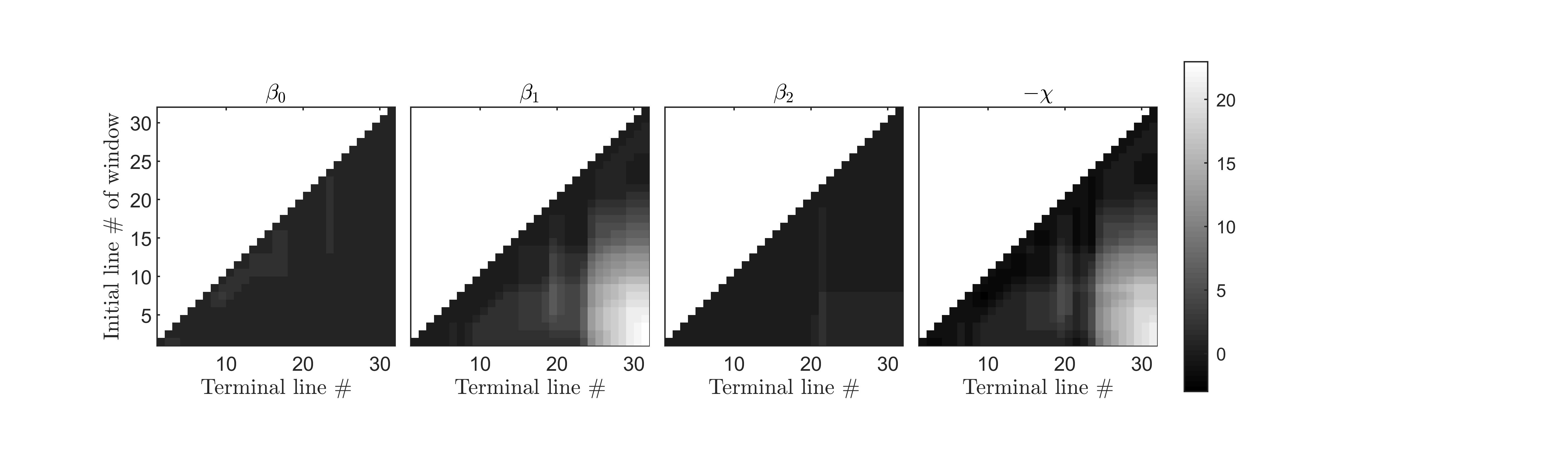}
\caption[Windowed Dowker homology for the \texttt{Makarov} $3 \times 3$ matrix multiplication algorithm]{Windowed Dowker homology for the \texttt{Makarov} $3 \times 3$ matrix multiplication algorithm. Lines 13-15 of the algorithm correspond to a local maximum in complexity. These lines embody a nontrivial homology class in dimension 1 corresponding to \texttt{A12}, \texttt{A22}, and \texttt{B23} that is isolated from the impact of lines 12 and 16 (i.e., there are no shared variables).
Lines 8-19 of the Makarov algorithm turn out to correspond to a local extremum in $\chi$ which is apparent in a thresholded version of this figure. The corresponding simplicial complex has 6 holes and 1 void (or “bubble”) that is not practical to visualize directly.}
\label{fig:Makarov}
\end{figure}


Finally, we can apply this same sort of construction at the binary level. In Figure \ref{fig:REIL}, we show a snippet of Reverse Engineering Intermediate Language (REIL) \cite{Dullien2009} code, the corresponding abstract simplicial complex (accounting for memory locations in a natural way), and the corresponding ``spectrograms.'' By limiting the size of windows considered, this sort of feature construction can be performed in linear time (albeit with a possibly large overhead constant) and used to analyze basic blocks in disassembled binaries or their rough equivalents.

%

\begin{figure}[htb]
\begin{minipage}{.29\textwidth}
\includegraphics[trim = 15mm 140mm 120mm 15mm, clip, keepaspectratio, width=\textwidth]{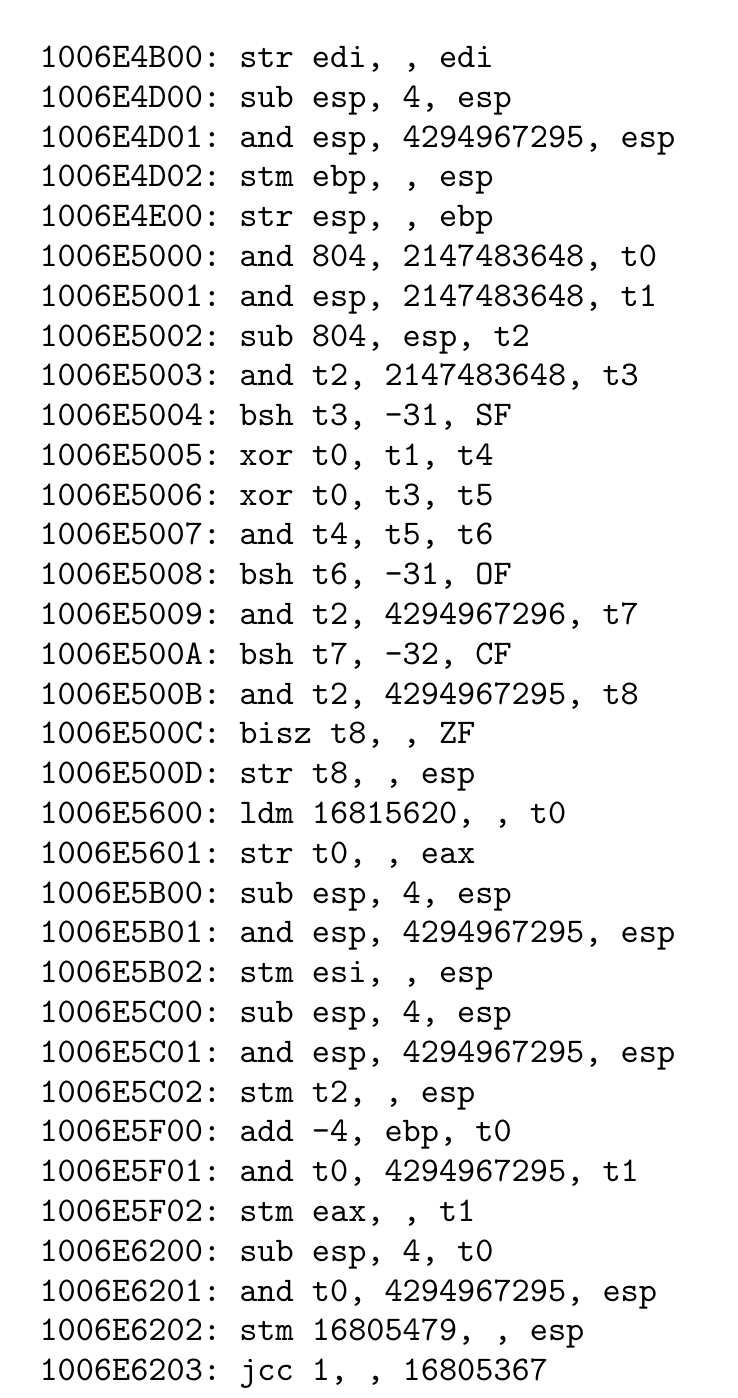}
\end{minipage}
\begin{minipage}{.69\textwidth}
\includegraphics[trim = 50mm 10mm 50mm 5mm, clip, keepaspectratio, width=\textwidth]{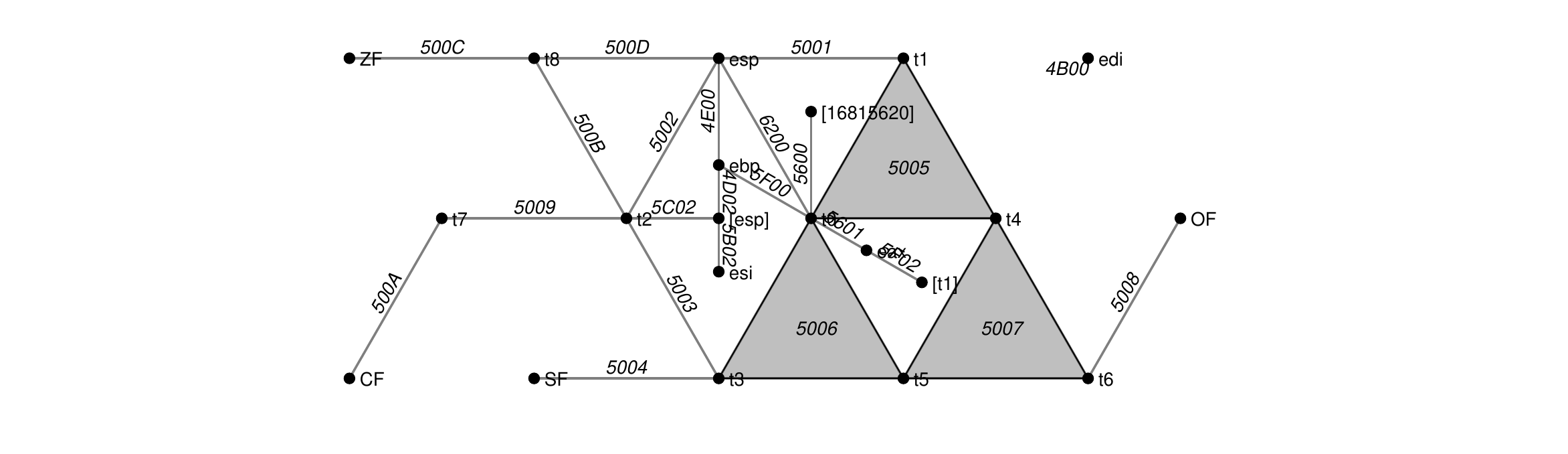} \\
\includegraphics[trim = 15mm 0mm 55mm 5mm, clip, keepaspectratio, width=\textwidth]{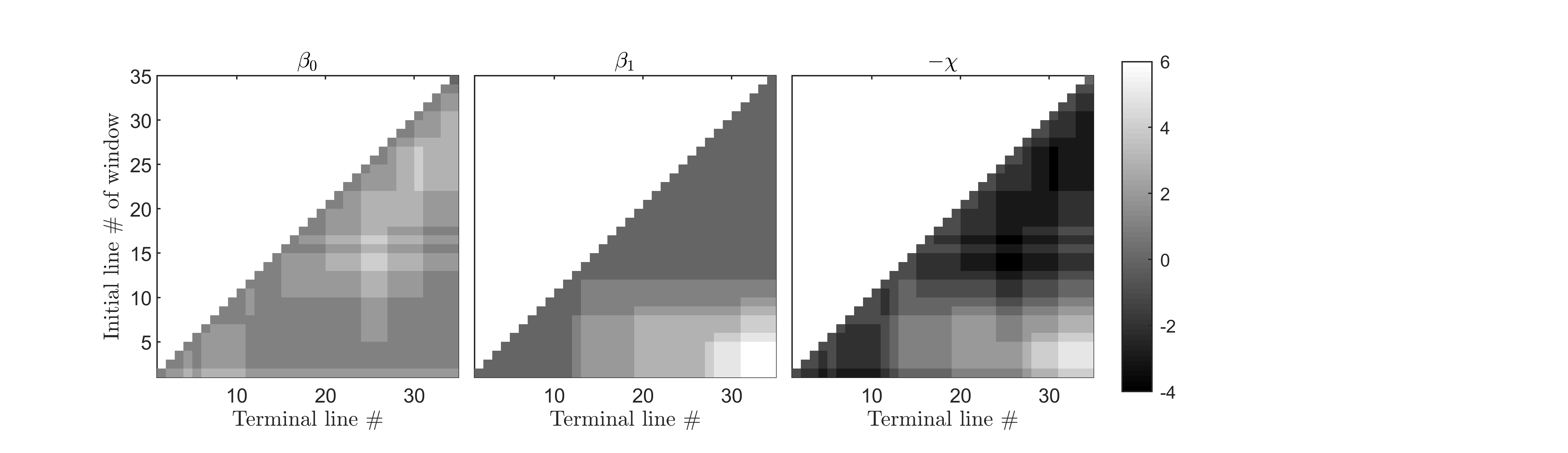}
\end{minipage}
\caption[Topology of an assembly-level basic block]{(L) Some REIL code. (R; upper) The corresponding 2-complex. (R; lower) A ``spectrogram'' of Betti numbers and Euler characteristic $\chi$ as a function of windowed code. Lines 7-13 (memory addresses ending in 5001-5007) exhibit clearly visible extrema in $\beta_1$ and $\chi$ (corner of a light rectangular region in center and right panels). Indeed, the registers \texttt{esp}, \texttt{t0}, \dots, \texttt{t6} are all involved in multiple instructions in this range, leading to a single connected component with two holes. A secondary locus of topological complexity is lines 11-13 (memory addresses ending in 5005-5007).
}
\label{fig:REIL}
\end{figure}

\section{Path homology to analyze graphical structures}\label{Path}

In this section, we introduce what turns out to be a generalization of many of the ideas in the preceding one, though we treat it on its own. Instead of computing topological invariants of shape-like data structures, we compute topological invariants of oriented path-like data structures. Because these are ubiquitous in cyber applications, we do not attempt an exhaustive treatment, but instead limit ourselves to sketching an application to the control flow of computer programs.

We outline \emph{path homology}\index{homology!path} as treated in \cite{Grigoryan2012,Chowdhury2018}. For additional background on path homology, see the series of papers \cite{Grigoryan2014,Grigoryan2014b,Grigoryan2015,Grigoryan2017,Grigoryan2018,Grigoryan2018b}.

For convenience, we replace the chain complex \eqref{eq:chainComplex} with its \emph{reduction}
\begin{equation}
\label{eq:reduction}
\dots C_{p+1} \overset{\partial_{p+1}}{\longrightarrow} C_p \overset{\partial_p}{\longrightarrow} C_{p-1} \overset{\partial_{p-1}}{\longrightarrow} \dots \overset{\partial_1}{\longrightarrow} C_0 \overset{\tilde \partial_0}{\longrightarrow} \mathbb{F} \longrightarrow 0
\end{equation}
which (using an obvious notational device and assuming the original chain complex is nondegenerate) has the minor effect $\tilde H_0 \oplus \mathbb{F} \cong H_0$, while $\tilde H_p \cong H_p$ for $p > 0$. Similarly, $\tilde \beta_p = \beta_p - \delta_{p0}$, where $\delta_{jk} = 1$ if and only if $j = k$ and $\delta_{jk} = 0$ otherwise.

For a loopless digraph $D = (V,A)$, the set $\mathcal{A}_p(D)$ of \emph{allowed $p$-paths} is
\begin{equation}
\label{eq:allowed}
\{(v_0,\dots,v_p) \in V^{p+1} : (v_{j-1},v_j) \in A, 1 \le j \le p\}.
\end{equation}
As a convention, we set $\mathcal{A}_{0} := V$, $V^0 \equiv \mathcal{A}_{-1} := \{0\}$ and $V^{-1} \equiv \mathcal{A}_{-2} := \emptyset$. For a field $\mathbb{F}$
\footnote{
Path homology can be defined over rings as well. This definition gives additional power: M. Yutin has exhibited digraphs on as few as six vertices that have torsion.
}
and a finite set $X$, let $\mathbb{F}^X \cong \mathbb{F}^{|X|}$ be the free $\mathbb{F}$-vector space on $X$, with the convention $\mathbb{F}^\emptyset := \{0\}$. The \emph{non-regular boundary operator} $\partial_{[p]} : \mathbb{F}^{V^{p+1}} \rightarrow \mathbb{F}^{V^p}$ is the linear map acting on the standard basis as
\begin{equation}
\label{eq:preboundary}
\partial_{[p]} e_{(v_0,\dots,v_p)} = \sum_{j=0}^p (-1)^j e_{\nabla_j (v_0,\dots,v_p )}.
\end{equation}
It is not hard to verify that $\partial_{[p-1]} \circ \partial_{[p]} \equiv 0$, so $(\mathbb{F}^{V^{p+1}},\partial_{[p]})$ is a chain complex.

Path homology is obtained from a different chain complex derived from the immediately preceding one. Set
\begin{equation}
\label{eq:invariant}
\Omega_p := \left \{ \omega \in \mathbb{F}^{\mathcal{A}_{p}} : \partial_{[p]} \omega \in \mathbb{F}^{\mathcal{A}_{p-1}} \right \},
\end{equation}
$\Omega_{-1} := \mathbb{F}^{\{0\}} \cong \mathbb{F}$, and $\Omega_{-2} := \mathbb{F}^{\emptyset} = \{0\}$. We have that $\partial_{[p]} \Omega_p \subseteq \mathbb{F}^{\mathcal{A}_{p-1}}$, so $\partial_{[p-1]} \partial_{[p]} \Omega_p = 0 \in \mathbb{F}^{\mathcal{A}_{p-2}}$ and $\partial_{[p]} \Omega_p \subseteq \Omega_{p-1}$. We can therefore define the \emph{(non-regular) path complex} of $D$ as the chain complex $(\Omega_p,\partial_p)$, where $\partial_p := \partial_{[p]}|_{\Omega_p}$.
\footnote{
The implied \emph{regular path complex} prevents a directed 2-cycle from having nontrivial 1-homology. While \cite{Grigoryan2012} advocates regular path homology, in our view non-regular path homology is simpler, richer, and more likely useful in applications.
}
The homology of this path complex is the \emph{(non-regular) path homology} of $D$.

For example, consider the digraphs $D_1$ and $D_2$ in Figure \ref{fig:phDigraphs}.  $\mathcal{A}_1(D_1)$ and $\mathcal{A}_1(D_2)$ are given by the directed edges, $\mathcal{A}_2(D_2) = \emptyset$, and  $\mathcal{A}_2(D_2) = \{(w,x,z), (w,y,z)\}$. Now $\partial_{[2]}e_{(w,x,z)} = e_{(x,z)} - e_{(w,z)} + e_{(w,x)} \not\in \mathbb{F}^{\mathcal{A}_1(D_2)}$ and $\partial_{[2]}e_{(w,y,z)} = e_{(y,z)} - e_{(w,z)} + e_{(w,z)} \not\in \mathbb{F}^{\mathcal{A}_1(D_2)}$ (because the edge $w \rightarrow z$ is missing), so
\begin{align}
\partial_{[2]}(e_{(w,x,z)} - e_{(w,y,z)}) & = e_{(x,z)} - e_{(w,z)} + e_{(w,x)} - e_{(y,z)} + e_{(w,z)} - e_{(w,y)} \nonumber \\ 
& = e_{(x,z)} + e_{(w,x)} - e_{(y,z)} - e_{(w,y)} \in \mathbb{F}^{\mathcal{A}_1(D_2)}. \nonumber
\end{align}
Consequently the dimensions of the path homology vector spaces (denoted by the Betti numbers $\beta$) are different: $\beta_1(D_1) = 1$ and $\beta_1(D_2) = 0$. 

\begin{figure}[htb]
\centering
\includegraphics[trim =100mm 130mm 90mm 125mm, clip, keepaspectratio, width=.2\textwidth]{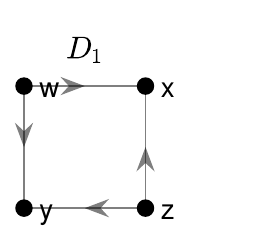}
\includegraphics[trim =100mm 130mm 90mm 125mm, clip, keepaspectratio, width=.2\textwidth]{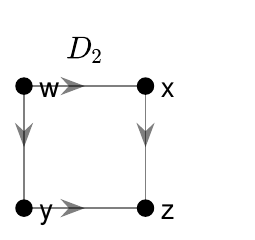}
\caption[Similar digraphs with different path homologies]{The digraph $D_1$ has trivial path homology but the digraph $D_2$ does not.}
\label{fig:phDigraphs}
\end{figure}


The ubiquity of digraphs in the cyber domain suggests that path homology can find a multitude of applications, and we briefly mention a few. 

Figure \ref{fig:intel_cfg} shows a control flow graph with nontrivial path homology in dimension two. It turns out that it is possible to construct control flow graphs (at the assembly level) with arbitrary path homology, and experiments suggest that path homology generalizes cyclomatic complexity in a way that can detect unstructured control flow \cite{Huntsman2020}.
\begin{figure}[htbp]
\begin{center}
\includegraphics[trim = 0mm 0mm 0mm 0mm, clip, width=.875\textwidth,keepaspectratio]{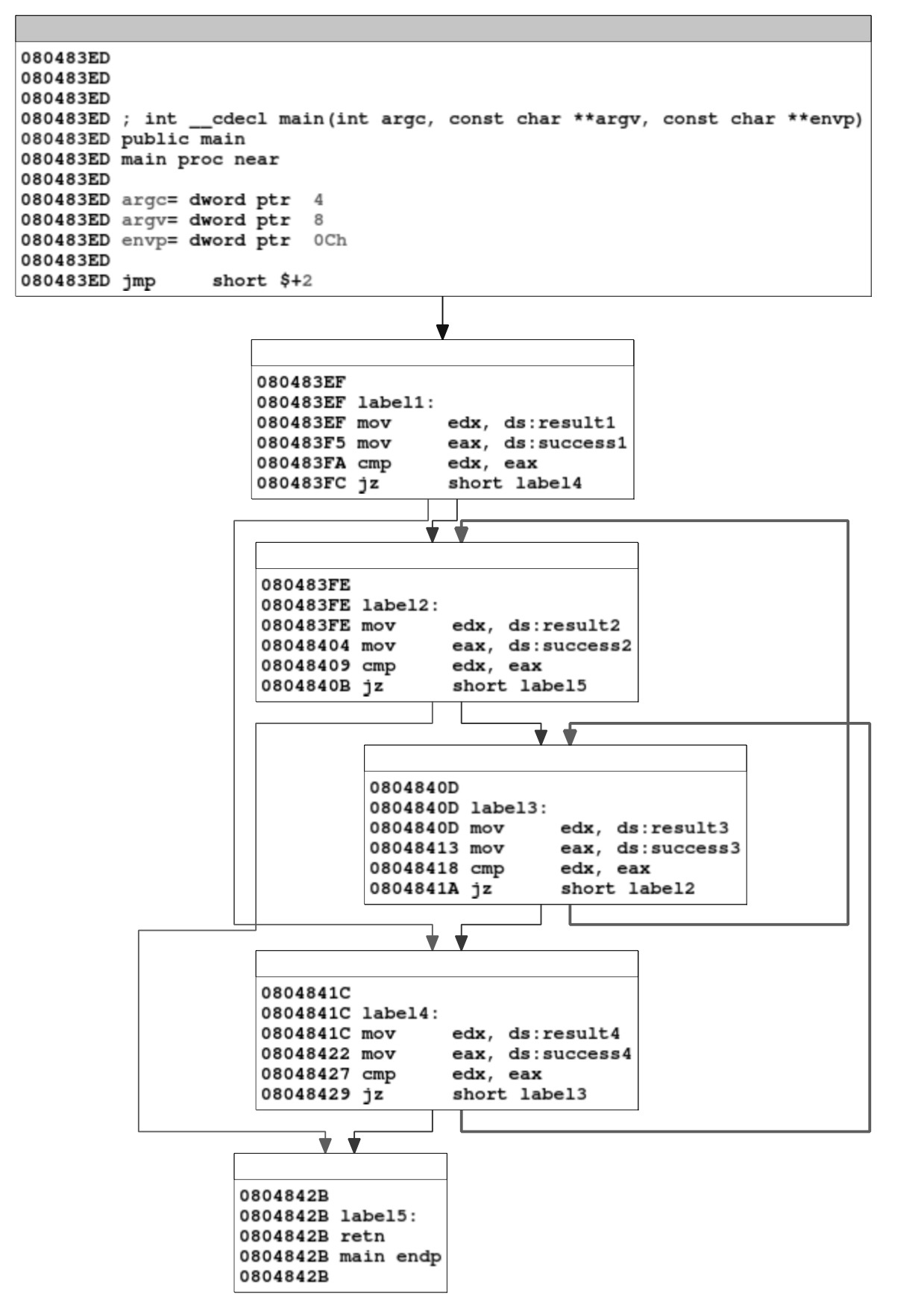}
\end{center}
\caption[A control flow graph with nontrivial path homology in dimension two]{A control flow graph with $\tilde \beta_\bullet = (0,1,1,0,\dots)$, obtained by disassembly in IDA Pro \cite{Eagle2011}. The binary is directly compiled from C source (albeit with \texttt{goto}s and inline assembly). The common instruction motif in most of the basic blocks clearly indicates how to construct binaries with essentially arbitrary control flow. Note that inserting operations without control flow (e.g., arithmetic operations in the instruction set) and reindexing memory addresses at various points would leave the control flow unaffected.
}
\label{fig:intel_cfg} 
\end{figure} %
The proof that control flow graphs can exhibit arbitrary path homology follows from a result of \cite{Chowdhury2019}, which itself has more direct applications to the characterization of neural networks. 

Meanwhile, the first author's analyses of UK and global air transportation networks (to be reported in a forthcoming paper) suggest that changes in the path homology of ``backbone'' digraphs (obtained by retaining only arcs corresponding to passenger volume above a threshold) as a function of the backbone threshold are strongly correlated with measures such as betweenness centrality. That is, path homology may provide network metrics that simultaneously complement and correlate with existing metrics.

\section{Topological data analysis and unsupervised learning in one dimension}\label{TDA1D}

In this section, we sketch the basic ideas of the rapidly expanding field of topological data analysis by considering a simple application in one dimension that simultaneously advances the state of the art in the fundamental area of nonparametric statistical estimation and avoids much of the technical baggage of persistent homology.

\emph{Topological data analysis}\index{topological data analysis} (TDA) has had a profound effect on data science and statistics over the last 15 years. Perhaps the most widely recognized and utilized tool in TDA is \emph{persistent homology}\index{persistent homology} \cite{zomorodian2005topology,Ghrist_2008,Carlsson2009,Edelsbrunner2010,GhristEAT,Oudot2015}. The basic idea (Figure \ref{fig:TDA}) is to associate an inclusion-oriented family (i.e., a \emph{filtration}\index{filtration}) of simplicial complexes to a point set in a metric space. Each simplicial complex in the filtration is formed by considering the intersections of balls of a fixed radius about each data point. As the radius varies, different simplicial complexes are produced, and their homologies are computed. 

\begin{figure}[htb]
\includegraphics[trim = 15mm 43mm 15mm 27mm, clip, keepaspectratio, width=\textwidth]{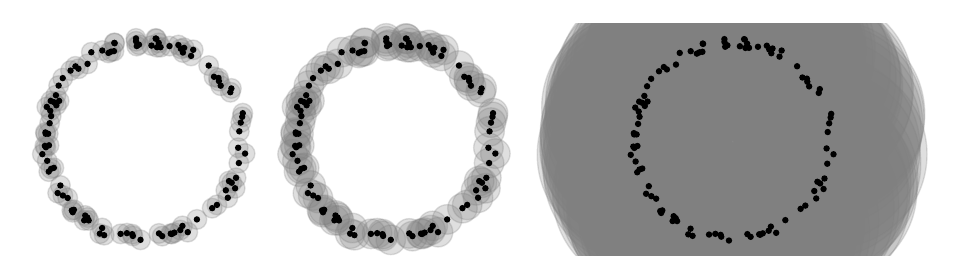}
\caption[Topological persistence]{The topology of a data set can be probed at different scales. Here, we consider a sample of 100 uniformly distributed points in a thin annulus about the unit circle. From left to right, we place disks of radius $0.1$, $0.15$, and $0.95$ around each point. The topology of the data set is morally that of a circle, and the (persistent) homology of simplicial complexes formed from the intersections of disks reveals this: a 1-homology class ``persists'' over an interval slightly bigger than $[.15, .95]$.}
\label{fig:TDA}
\end{figure}

Although the theory of topological persistence involves a considerable amount of algebra for bookkeeping associated to the ``births'' and ``deaths'' of homology classes as a function of the radius/filtration parameter, in practice simply treating the Betti numbers as functions of that parameter gives considerable information. Along similar lines, we can consider how other topological invariants behave as a function of scale. 

\begin{figure}[htb]
\begin{center}
\includegraphics[trim = 40mm 109mm 40mm 100mm, clip, keepaspectratio, width=.48\textwidth]{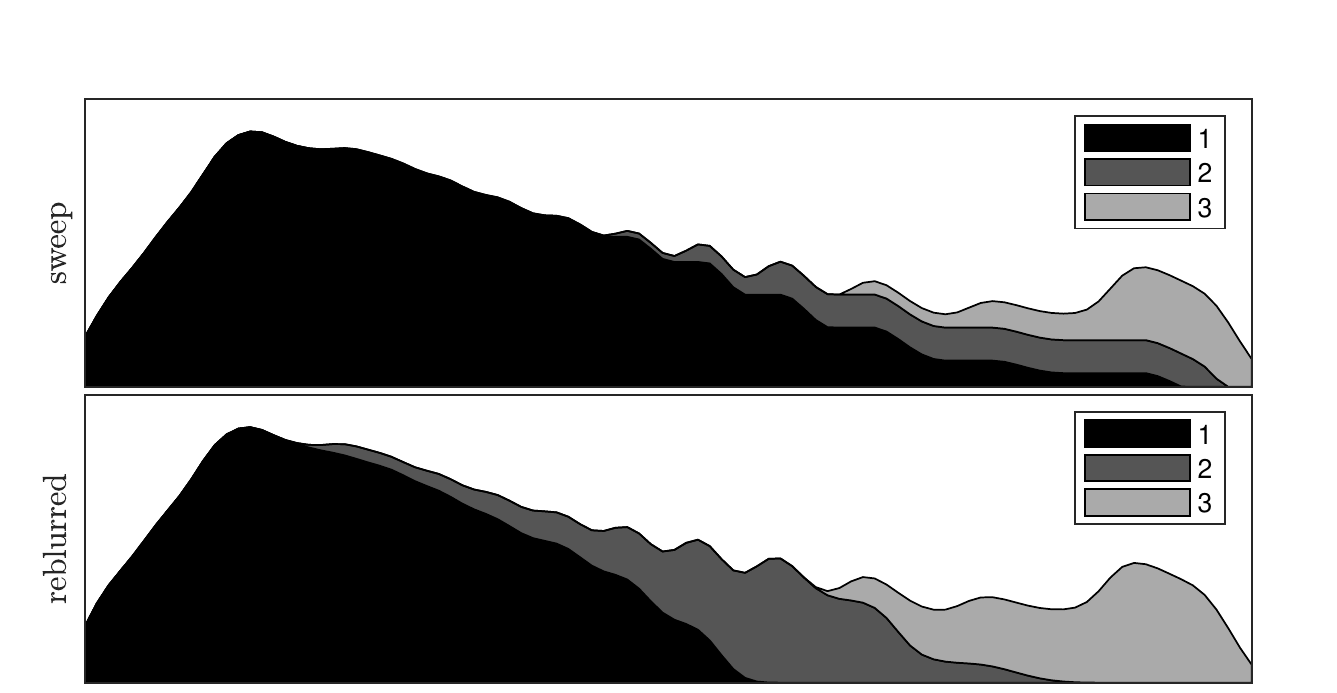}
\includegraphics[trim = 40mm 109mm 40mm 100mm, clip, keepaspectratio, width=.48\textwidth]{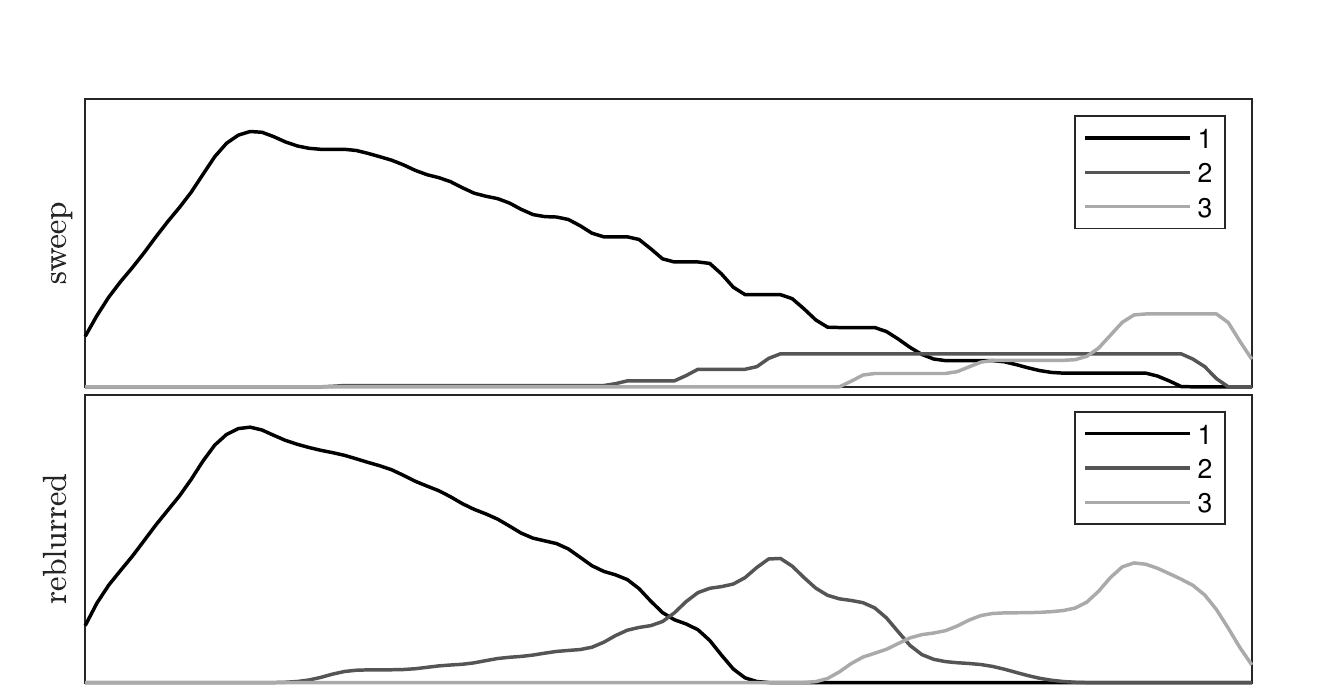}
\end{center}
\caption[Topological mixture estimation]{Topological mixture estimation. Left panels: area plots of (top) initial and (bottom) information-theoretically optimized unimodal decompositions of an estimated probability distribution. Right panels: line plots of the same decompositions. The bandwidth for the kernel density estimate \footnote{A kernel density estimate is a sort of ``smooth histogram'' that represents sample data by an average of copies of a ``kernel'' probability distribution centered around the data points. The bandwidth of a kernel density estimate is a scaling factor for the kernel: small bandwidths scale the kernel to be narrow and tall, and large bandwidths scale the kernel to be wide and short.} for the distribution and the number of unimodal mixture components are both determined using the same topological considerations.}
\label{fig:TME}
\end{figure}

Call $\phi : \mathbb{R}^n \rightarrow [0,\infty)$ \emph{unimodal}\index{unimodal} if $\phi$ is continuous and the excursion set $\phi^{-1}([y,\infty))$ is contractible (i.e., homotopy equivalent to a point) for all $0 < y \le \max \phi$. For $n = 1$, contractibility means that these excursion sets are all intervals, which coincides with the intuitive notion of unimodality. For $f : \mathbb{R}^n \rightarrow [0,\infty)$ sufficiently nice, define the \emph{unimodal category}\index{unimodal!category} of $f$ to be the smallest number $M$ of functions such that $f$ admits a \emph{unimodal decomposition}\index{unimodal!decomposition} of the form $f = \sum_{m=1}^M \pi_m \phi_m$ for some $\pi > 0$, $\sum_m \pi_m = 1$, and $\phi_m$ unimodal \cite{GhristEAT}. 

The unimodal category is a topological (homeomorphism) invariant and a ``sweep'' algorithm due to Baryshnikov and Ghrist efficiently produces a unimodal decomposition in $n = 1$. \footnote{The case $n = 2$ is still beyond the reach of current techniques, and only partial results are known. Moreover, for $n$ sufficiently large, there is provably no algorithm for computing the unimodal category!} As Figure \ref{fig:TME} demonstrates, the unimodal category can be much less than the number of extrema.

The unimodal category of a kernel density estimate for a probability distribution can be used to select an appropriate bandwidth for sample data and, as shown in Figure \ref{fig:TME}, to decompose the resulting estimated distribution into well-behaved unimodal components \emph{using no externally supplied parameters whatsoever} \cite{Huntsman2018}. The key ideas behind \emph{topological mixture estimation}\index{topological mixture estimation} are to identify the most common unimodal category as a function of bandwidth and to exploit convexity properties of the mutual information between the mixture weights and the distribution itself. The result is an extremely general (though also computationally expensive) unsupervised learning technique in one dimension that can, e.g. automatically set thresholds for anomaly detectors or determine the number of clusters in data (by taking random projections).

\section{Critical node detection in wireless networks using sheaves}\label{Sheaves}


The abstract simplicial complex tools developed in the previous sections of this chapter can also be applied to understand the structure of wireless communication networks.  As before, the combinatorial nature of such a network aligns neatly with the combinatorial structure of an abstract simplicial complex.  Qualitative intuition about how the network responds to stress can be transformed into quantitative analytic tools using the topology of these simplicial complexes.

When a carrier sense multiple access/collision detection (CSMA/CD) media access model is used in a wireless network, only one node in a given vicinity can transmit while the others must wait.  Although the physical layer protocols of wireless networks can be quite complex, the basic topology of the network plays an important role in determining network performance.  This section addresses the problem of identifying critical nodes and links within a network by using local invariants derived from the local topology of the network.  Recognizing that although protocol plays an important role, we are specifically concerned with those effects that are \emph{protocol independent}.

This section provides theoretical justification for the ``right'' local neighborhood in a wireless network with a CSMA/CD media access model using the structure of network activation patterns, and then validates the resulting topological invariants using simulated network traffic generated with {\tt ns2}.  

\subsection{Historical context and contributions}

Graph theory methods have been used extensively (for instance \cite{Nandagopal_2000,Yang_2002,Jain_2003,Lee_2007}) for identifying critical nodes in a network that carry a disproportionate amount of traffic.  However, direct application of graph theory to locate these nodes is computationally expensive \cite{DiSumma_2011,dinh2012}.  Furthermore, graphs are better suited to \emph{wired} networks and don't necessarily address the multi-way interactions inherent in wireless networks \cite{Chiang_2007}.

We can extend the ideas discussed earlier in this chapter to wireless networks by using higher-dimensional abstract simplicial complexes instead of graph connectivity as a measure of network health.  Although connectivity can be a useful measure of health \cite{Noubir_2004,Gueye_2010}, it is rather coarse.  We remedy this with a more systematic study of an 802.11b wireless network using the {\tt ns2} network simulator \cite{nsnam}.

\subsection{Interference from a transmission}

One of the main differences between a wired and a wireless communication network is the prevalence of interference on shared channels.  Channels that are shared by more users or nodes are more likely to be congested.  An abstract simplicial complex called the interference complex can model the shared channel usage within a wireless network, and forms the basis of its topological analysis.

Let a wireless network consist of a single channel, with nodes $N=\{n_1,n_2, \dotsc, n_i, \dotsc\}$ in a region $R$.  Associate an open set $U_i \subset R$ to each node $n_i$ that represents its \emph{transmitter coverage region}.  For each node $n_i$, a continuous function $s_i : U_i \to \mathbb{R}$ represents its \emph{signal level} at each point in $U_i$.  Without loss of generality, we assume that there is a global threshold $T$ for accurately decoding the transmission from any node.  In \cite{RobinsonGlobalSIP2014}, two abstract simplicial complex models were developed: the \emph{interference} and \emph{link} complexes.

\begin{definition}
The \emph{interference complex}\index{interference complex} is the abstract simplicial complex $I=I(N,U,s,T)$ consisting of all subsets of $N$ of the form $[n_{i_1}, \dotsc, n_{i_m}]$ for which $U_{i_1} \cap \dotsb \cap U_{i_m}$ contains a point $x \in R$ for which $s_{i_k}(x) > T$ for all $k=1, \dotsb m$.  
\end{definition}
The vertices of the interference complex are the nodes $N$ of the network.  There is a simplex for each list of transmitters that when transmitting will result in at least one mobile receiver location receiving multiple signals simultaneously.  (The interference complex is a \v{C}ech complex \cite{GhristEAT, Hatcher_2002}.)

\begin{proposition}
Each facet of the interference complex corresponds to a maximal collection of nodes that mutually interfere.
\end{proposition}
\begin{proof}
Let $c$ be a simplex of the interference complex.  Then $c$ is a collection of nodes whose coverages have a nontrivial intersection.  The decoding threshold is exceeded for all nodes at some point $x$ in this intersection.  If any two nodes in $c$ transmit simultaneously, they will interfere at $x$.  If $c$ is a facet, it is contained in no larger simplex, so it is clearly maximal.
\end{proof}

\begin{definition}
The \emph{link graph}\index{link graph} is a $1$-dimensional simplicial complex defined by the following collection of subsets of $N$:
\begin{enumerate}
\item $[n_i] \in N$ for each node $n_i$, and
\item $[n_i,n_j]\in N$ if $s_i(n_j) > T$ and $s_j(n_i) >T$.
\end{enumerate}
The \emph{link complex} $L=L(N,U,s,T)$ is the clique complex of the link graph, which means that it contains all elements of the form $[n_{i_1},\dotsc,n_{i_m}]$ whenever this set is a clique in the link graph.
\end{definition}

\begin{figure}
\begin{center}
\includegraphics[width=1.0\textwidth]{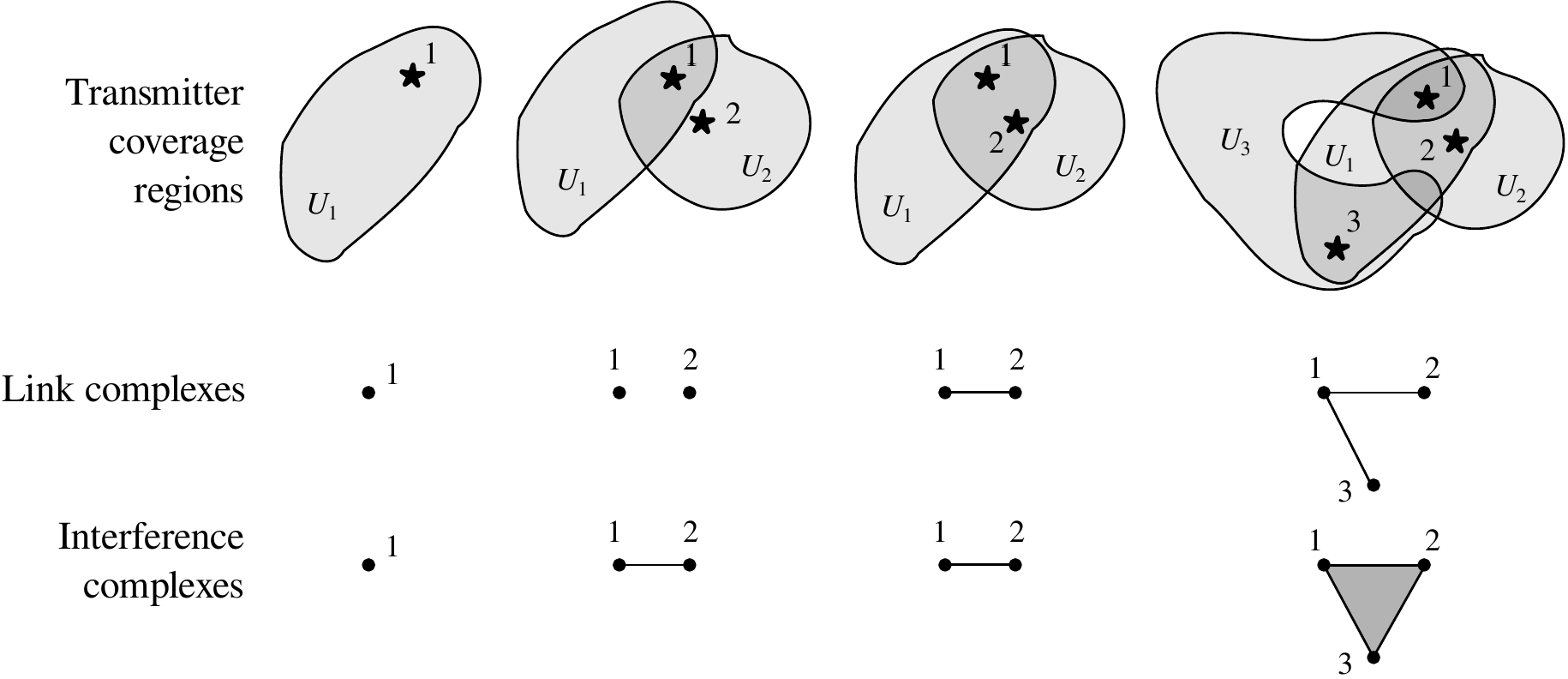}
\caption{Several transmitters marked by stars, their coverage regions (top row), their link complexes (middle row), and their interference complexes (bottom row).}
\label{fig:network_complexes}
\end{center}
\end{figure}

Figure \ref{fig:network_complexes} shows three transmitters, labeled $1$, $2$, and $3$, with their coverage regions $U_1$, $U_2$, and $U_3$ for a particular threshold $T$.  Assuming that all points within $U_i$ can receive the signal from transmitter $i$, the link complex for each configuration is shown in the middle row of Figure \ref{fig:network_complexes}.  Notice that in the second column, transmitter $1$ can receive transmitter $2$'s signal but not conversely.  This explains the absence of an edge in the link complex.  However, since there are points in the intersection between their two coverage regions, the interference complex contains an edge.  This also happens in the rightmost column, in which neither of transmitter $2$ or $3$ can receive each other's signal, but there are points where all three transmitters can be received.

\begin{proposition}
Each facet in the link complex is a maximal set of nodes that can communicate directly with one another (with only one transmitting at a time).
\end{proposition}
\begin{proof}
Let $c$ be a simplex of the link complex.  By definition, for each pair of nodes, $i,j\in c$ implies that $s_i(n_j) > T$ and $s_j(n_i) >T$.  Therefore, $i$ and $j$ can communicate with one another.  
\end{proof}

\begin{corollary}
Facets of the link complexes represent common broadcast resources.
\end{corollary}

Since the CSMA/CD protocol is implemented locally, it can be modeled as follows:

\begin{definition}\index{activation sheaf}
Suppose that $X$ is a simplicial complex (such as an interference or link complex) whose set of vertices is $N$.  Consider the following assignment $\mathcal{A}$ of additional information to capture which nodes are transmitting and decodable:
\begin{enumerate}
\item To each simplex $c\in X$, assign the set
\begin{eqnarray*}
\mathcal{A}(c)&=&\{n \in N : \text{there exists a simplex } d\in X \text{ with }\\
&& c \subset d \text{ and } n\in d\}\cup\{\perp\}
\end{eqnarray*} 
of nodes that have a coface in common with $c$, along with the symbol $\perp$.  We call $\mathcal{A}(c)$ the \emph{stalk} of $\mathcal{A}$ at $c$.
\item To each pair $c \subset d$ of simplices, assign the \emph{restriction function}
\begin{equation*}
\mathcal{A}(c\subset d)(n) =
\begin{cases}
n&\text{if }n\in \mathcal{A}(d)\\
\perp&\text{otherwise}\\
\end{cases}
\end{equation*}
\end{enumerate}
\end{definition}

For instance, if $c \in X$ is a simplex of a link complex, $\mathcal{A}(c)$ specifies which nearby nodes are transmitting and decodable, or $\perp$ if none are.  The restriction functions relate the decodable transmitting nodes at the nodes to which nodes are decodable along an attached wireless link.  Similarly, if $c \in X$ is a simplex of an interference complex, $\mathcal{A}(c)$ also specifies which nearby nodes are transmitting, and effectively locks out any interfering transmissions from other nodes.  

\begin{definition}
  The assignment $\mathcal{A}$ is called the \emph{activation sheaf} and is a sheaf on an abstract simplicial complex.
\end{definition}

  The theory of sheaves explains how to extract consistent information called \emph{sections}, which in the present context consists of nodes whose transmissions do not interfere with one another.

  \begin{definition}
A \emph{section}\index{section!of a sheaf} of $\mathcal{A}$ supported on a subset $Y \subseteq X$ is a function $s:Y \to N$ so that for each $c \subset d$ in $Y$, $s(c) \in \mathcal{A}(c)$ and $\mathcal{A}(c\subset d)\left( s(c)\right) = s(d)$.  We call the subsset $Y$ the \emph{support}\index{support!of a section} of the section.  A section supported on $X$ is called a \emph{global section}.  
\end{definition}

Specifically, global sections are complete lists of nodes that can be transmitting without interference.  

\begin{figure}
\begin{center}
\includegraphics[width=2.5in]{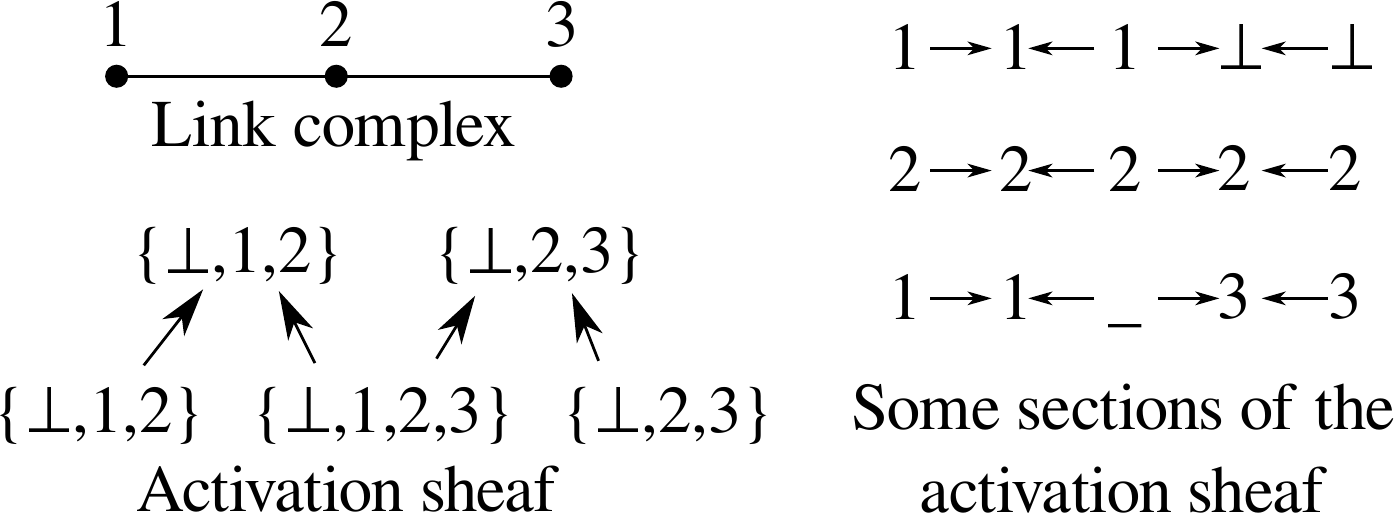}
\caption{A link complex (left top), sheaf $\mathcal{A}$ (left bottom), and three sections (right).  The restrictions are shown with arrows. There is a global section when node 1 transmits (right top), a global section when node 2 transmits (right middle), and a local section with nodes 1 and 3 attempting to transmit, interfering at node 2 (right bottom).  An underscore in the right bottom frame indicates where an element is outside the support of the section.}
\label{fig:linesec}
\end{center}
\end{figure}

Figure \ref{fig:linesec} shows a network with three nodes, labeled 1, 2, and 3.  When node 1 transmits, node 2 receives.  Because node 2 is busy, its link to node 3 must remain inactive (right top).  When node 2 transmits, both nodes 1 and 3 receive (right middle).  The right bottom diagram shows a local section that cannot be extended to the simplex marked with a blank.  This corresponds to the situation where nodes 1 and 3 attempt to transmit but instead cause interference at node 2. 

\begin{definition}
Suppose that $s$ is a global section of $\mathcal{A}$.  The \emph{active region}\index{active region} associated to a node $n\in X$ in $s$ is the set
\begin{equation*}
\text{active}(s,n) = \{a \in X : s(a)=n\},
\end{equation*}
which is the set of all nodes that are currently waiting on $n$ to finish transmitting.
\end{definition}

\begin{lemma}
\label{lem:act}
The active region of a node is a connected, closed subset of $X$ that contains $n$.
\end{lemma}
\begin{proof}
Consider a simplex $c \in \text{active}(s,n)$.  If $c$ is not a vertex, then there exists a $b \subset c$; we must show that $b\in \text{active}(s,n)$.  Since $s$ is a global section $\mathcal{A}(b\subset c)s(b)=s(c)=n$.  Because $s(c) \not= \perp$, the definition of the restriction function $\mathcal{A}(b\subset c)$ implies that $s(b)=n$.  Thus $b\in \text{active}(s,n)$ so $\text{active}(s,n)$ is closed.

If $c \in \text{active}(s,n)$, then $c$ and $n$ have a coface $d$ in common.  Since $s$ is a global section $s(d)=\mathcal{A}(c \subset d)s(c)=\mathcal{A}(c \subset d)n=n$.  Thus, $n \in \text{active}(s,n)$, because $n$ is a face of $d$ and $\text{active}(s,n)$ is closed.  This also shows that every simplex in $\text{active}(s,n)$ is connected to $n$.
\end{proof}

\begin{lemma}
\label{lem:stactive}
The star over the active region of a node does not intersect the active region of any other node.
\end{lemma}
\begin{proof}
Let $c\in \text{star } \text{active}(s,n)$.  Without loss of generality, assume that $c \notin \text{active}(s,n)$.  Therefore, there is a $b \in \text{active}(s,n)$ with $b\subset c$.  By the definition of the restriction function $\mathcal{A}(b\subset c)$, the assumption that $c\notin \text{active}(s,n)$, and the fact that $s$ is a global section, $s(c)$ must be $\perp$.
\end{proof}

\begin{corollary}
  If $s$ is a global section of an activation sheaf $\mathcal{A}$, then the set of simplices $c$ where $s(c) \not= \perp$ consists of a disjoint union of active regions of nodes. 
\end{corollary}

\begin{lemma}
  \label{lem:activeinvariant}
The active region of a node is independent of the global section.  More precisely, if $r$ and $s$ are global sections of $\mathcal{A}$ and the active regions associated to $n \in X$ are nonempty in both sections, then $\text{active}(s,n)=\text{active}(r,n)$.
\end{lemma}

Notice that if either of $r$ or $s$ has an empty active region, then Lemma \ref{lem:activeinvariant} makes no assertions.

\begin{proof}
Without loss of generality, we need only show that $\text{active}(s,n) \subseteq \text{active}(r,n)$.  If $c \in\text{active}(s,n)$, there must be a simplex $d\in X$ that has both $n$ and $c$ as faces.  Now $s(n)=r(n)=n$ by Lemma \ref{lem:act}, which means that $r(d)=\mathcal{A}(n \subset d)r(n)=n$.  Therefore, since $\text{active}(r,n)$ is closed, this implies that $c \in \text{active}(r,n)$.  
\end{proof}

\begin{figure}
\begin{center}
\includegraphics[width=1.0\textwidth]{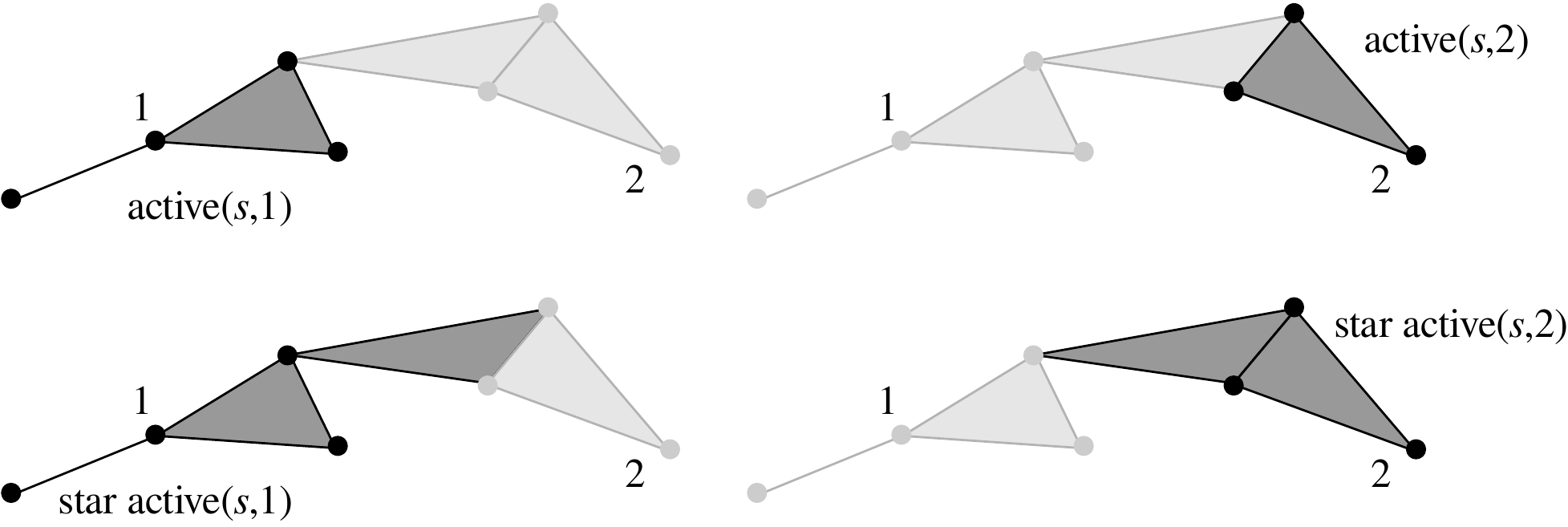}
\caption{The active regions of two transmitters within a link complex (top row) and the stars over their active regions (bottom row).}
\label{fig:active_regions}
\end{center}
\end{figure}

Figure \ref{fig:active_regions} shows an example of a link complex in which two transmitters, labeled $1$ and $2$, are indicated.  Their active regions are shown in the top row of Figure \ref{fig:active_regions}.  Because of Lemma \ref{lem:act}, each of these active regions is a closed set.  The stars over their active regions are shown in the bottom row of Figure \ref{fig:active_regions}.  Notice that because of Lemma \ref{lem:stactive}, the star over the active region of transmitter $1$ does not intersect the active region of transmitter $2$ or vice versa.  Additionally, according to Lemma \ref{lem:activeinvariant}, it is unnecessary to specify the global section of the activation sheaf used to construct these regions.

\begin{corollary}
  \label{cor:activationsections}
  The space of global sections of an activation sheaf consists of all sets of nodes that can be transmitting simultaneously without interference.
\end{corollary}

\subsection{An algebraic interlude: relative homology}

Homology is a global topological invariant, which is to say that it applies to the entirety of a topological space.  Since we wish to identify portions of the network that are more critical, it is useful to construct a local version of homology.  This can be achieved by an algebraic construction that temporarily removes a portion of the space from consideration, called relative homology.

Suppose that $Y \subseteq X$ is a subcomplex of an abstract simplicial complex.

\begin{definition}
  \label{def:chain_space}
The \emph{relative $k$-chain space}\index{chain complex!relative} $C_k(X,Y)$ is the vector space whose basis consists of the $k$-dimensional simplices of $X$ that are not in $Y$.  We can define the \emph{relative boundary map} $\partial_k : C_k(X,Y) \to C_{k-1}(X,Y)$ using
\begin{equation}
  \partial_k ([v_0,\dotsc,v_k]) = \sum_{j=0}^k (-1)^j
  \begin{cases}
    \nabla_j [v_0,\dotsc,v_k] &\text{if }\nabla_j  [v_0,\dotsc,v_k] \notin Y,\\
    0 & \text{otherwise.}\\
\end{cases}
\end{equation}
\end{definition}

This is really a more elaborate form of the simplicial chain complex defined in \eqref{eq:simplicial}, and the same proof as before establishes that $(C_\bullet(X,Y),\partial_\bullet)$ is a chain complex.  Naturally enough, there is a notion of \emph{relative simplicial homology}\index{simplicial homology!relative}.

\begin{definition}
  \label{def:relative_simplicial_homology}
  For a subcomplex $Y \subseteq X$ of an abstract simplicial complex $X$,
  \begin{equation*}
    H_k(X,Y):=H_k(C_\bullet(X,Y),\partial_\bullet)
  \end{equation*}
  is called the \emph{relative homology of the pair $(X,Y)$}.
\end{definition}

As before, there is a notion of simplicial maps inducing maps on the relative homology.  However, not every simplicial map works: it needs to respect the subcomplexes!

\begin{proposition}\cite[Props. 2.9, 2.19]{Hatcher_2002}
  \label{prop:relative_functor}
Every simplicial map $f: X \to Z$ from one abstract simplicial complex to another which restricts to a simplicial map $Y \to W$ induces a linear map $H_k(X,Y)\to H_k(Z,W)$ for each $k$.  We call $(X,Y)$ and $(Z,W)$ \emph{simplicial pairs} and $f$ a \emph{pair map} $(X,Y) \to (Z,W)$.
\end{proposition}

\subsection{Using activation patterns}

The structure of the global sections of an activation sheaf leads to a model in which an active node silences all other nodes in its vicinity.  In this section, we develop the concept of the local homology dimension, and show how it can identify topological ``pinch points'' within the network.

\begin{definition}
  Because of the Lemmas, we call the star over an active region associated to a node $n$ the \emph{region of influence}\index{region of influence}.  The region of influence of a facet is the star over the closure of that facet.  The region of influence for a collection of facets $F$ can be written as a union
\begin{equation*}
\text{roi } F = \bigcup_{f \in F} \text{star } \text{cl } f.
\end{equation*}
\end{definition}

One can therefore interpret the bottom row of Figure \ref{fig:active_regions} as showing the regions of influence of transmitters $1$ and $2$.

In our previous work \cite{RobinsonGlobalSIP2014}, the region of influence was used without detailed justification; the following Corollary provides this needed justification.

\begin{corollary}
\label{cor:unaffected}
The complement of the region of influence of a facet is a closed subcomplex.
\end{corollary}

Given this justification, \cite{RobinsonGlobalSIP2014} shows that critical nodes or links are those simplices $c$ for whom the \emph{local homology dimension}\index{homology!local} (see also \cite{Joslyn_2016})
\begin{equation*}
  LH_k(c) = \text{dim } H_k(X,X\,\backslash\,\text{roi } c)
\end{equation*}
is larger than the average.

This implies the following experimental hypothesis: \emph{If a node is critical, it will have a large local homology dimension}.  Since the {\tt ns2} network simulator provides complete transcripts of all packets, we can define a critical node to be one that \emph{forwards} a large number of packets compared to other nodes in the network \cite{Arulselvan_2009}.

\begin{figure}
\begin{center}
\includegraphics[width=4in]{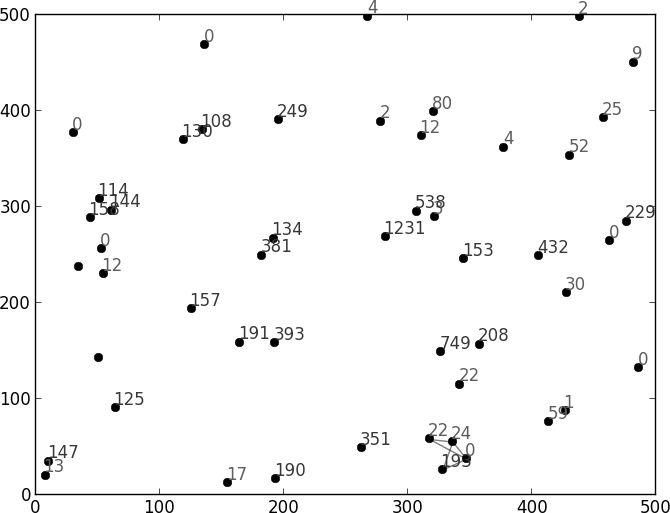} 
\caption{Locations of nodes and forwarded packet counts (axes in meters)}
\label{fig:network}
\end{center}
\end{figure}

We constructed a small simulation with 50 nodes as shown in Figure \ref{fig:network}.  Packets were randomly assigned source and destination nodes within the network, and all packet histories were recorded for analysis.

\begin{figure}
\begin{center}
\includegraphics[width=3in]{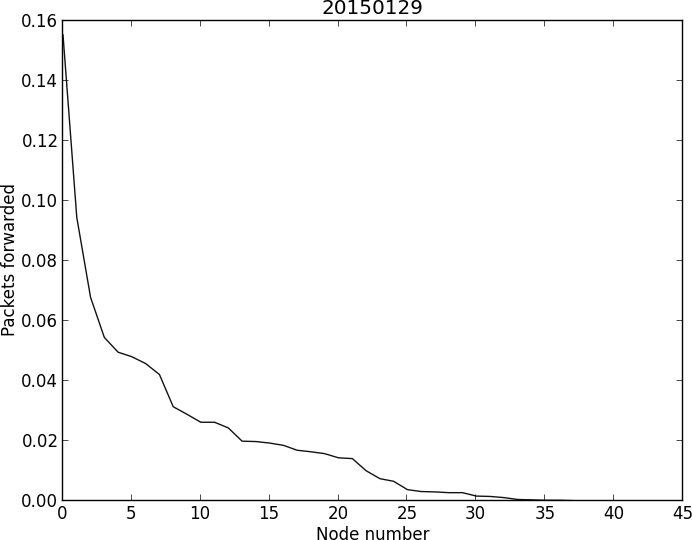} 
\caption{Probability that a given packet will be forwarded by a specific node}
\label{fig:fwdhist}
\end{center}
\end{figure}

Figure \ref{fig:fwdhist} shows the probability that a node will forward a random packet.  (The node numbers have been sorted from greatest to least probability.) The figure shows that most nodes forward only a small number of packets, while a few nodes carry considerably more traffic.

\begin{figure}
\begin{center}
\includegraphics[width=4in]{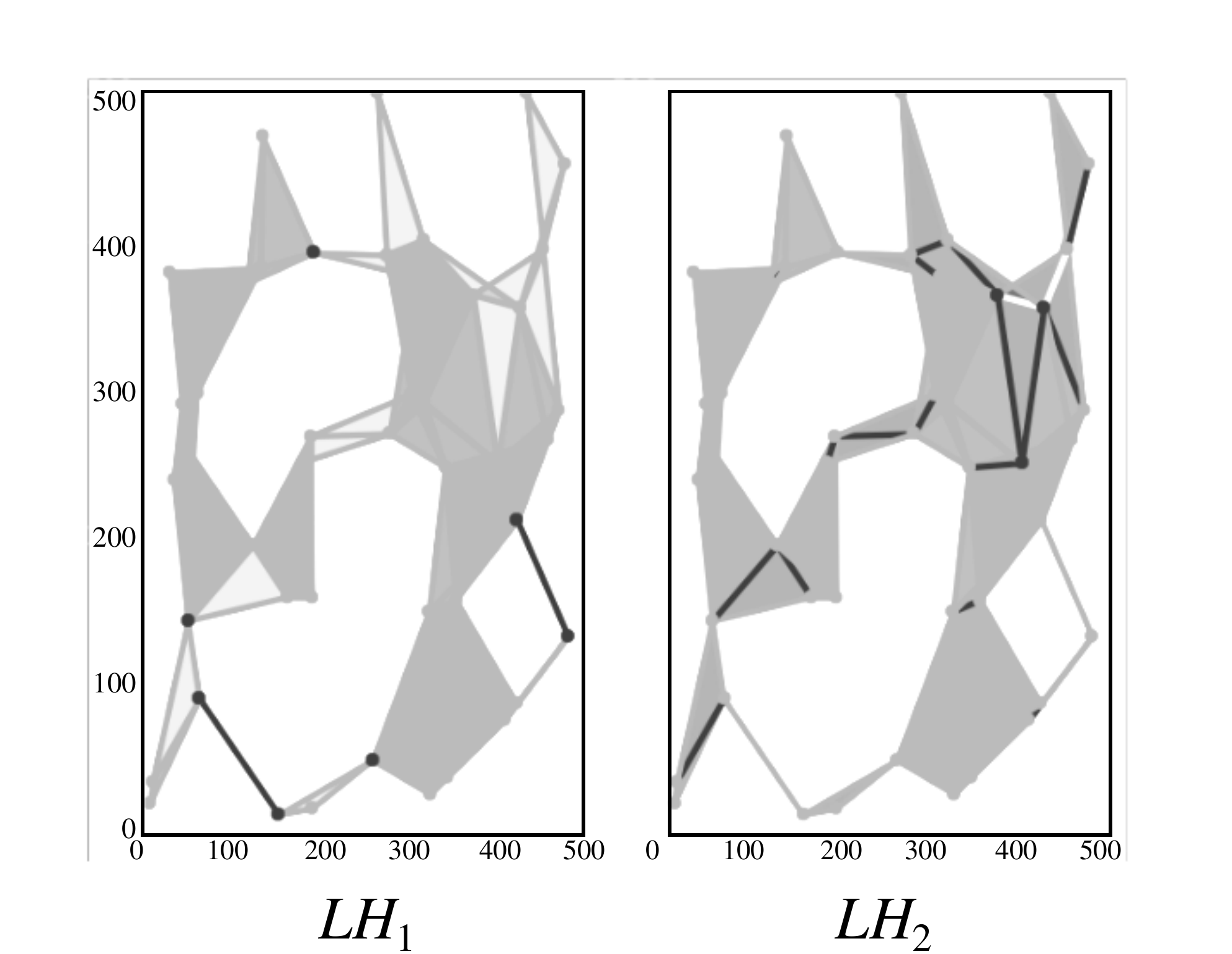} 
\caption{Dimension of local homology $LH_1$ (left) and $LH_2$ (right).  Axes in meters; Gray = 0, Black = 1, White = 2.}
\label{fig:networklh}
\end{center}
\end{figure}

Figure \ref{fig:networklh} shows the dimension of local homology over all nodes and links in the network.  In this particular network, the local homology dimension is only 0, 1, or 2.  It is clear that nodes with high $LH_1$ occupy certain ``pinch points'' in the network.

\begin{figure}
\begin{center}
\includegraphics[width=3.5in]{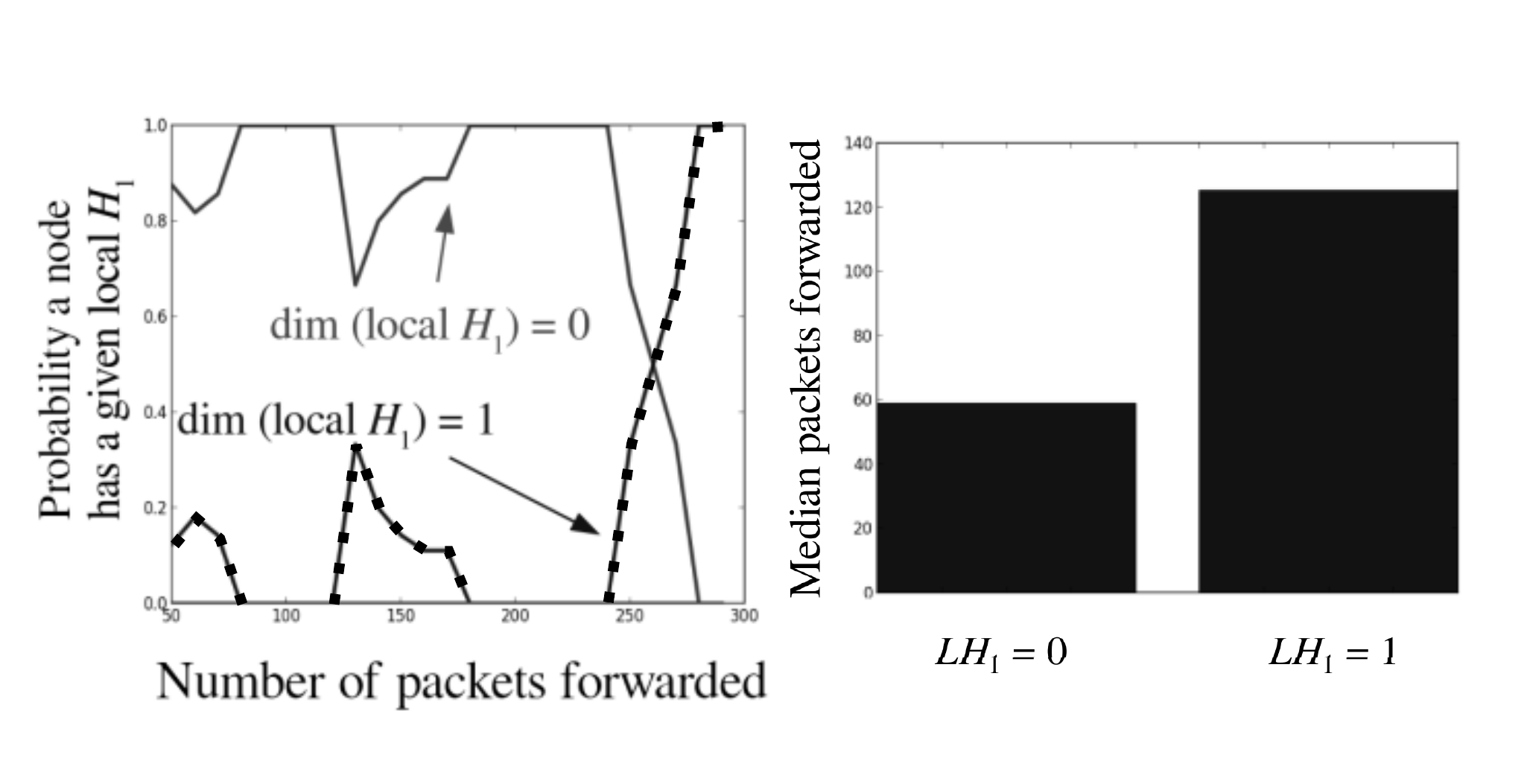} 
\caption{Probability a node has a certain local homology dimension given the number of packets it forwards}
\label{fig:fwdlh}
\end{center}
\end{figure}

Figure \ref{fig:fwdlh} shows the probability that a node forwarding a certain number of packets will have the given value of $LH_1$.  (We did not find a strong correspondence between forwarded packets and $LH_2$.)  It is immediately clear that all nodes forwarding a large number of packets are assigned a high local homology, but the converse is not necessarily true.  Local homology dimension is an indication that a node may be critical, but does not guarantee criticality.

\subsection{Cohomological analysis}

Although the space of global sections for an activation sheaf is a useful invariant, its sheaf cohomology\footnote{For background on, and other practical applications of, sheaf cohomology, see \cite{Robinson_2014, GhristEAT}.} is rather uninteresting.  We need to enrich their structure somewhat to see this, though.

\begin{definition}
  If $\mathcal{A}$ is an activation sheaf on an abstract simplicial complex $X$, the \emph{vector activation sheaf}\index{activation sheaf!vector} $\widehat{\mathcal{A}}$ is given by specifying its stalks and restrictions:
\begin{enumerate}
\item To each simplex $c\in X$, let $\widehat{\mathcal{A}}(c)$ be the vector space whose basis is $\mathcal{A}\backslash\{\perp\}$ (so the dimension of this vector space is the cardinality of $\mathcal{A}$ without counting $\perp$)
\item The restriction map $\widehat{\mathcal{A}}(c\subset d)(n)$ is the basis projection, which is well-defined since $\mathcal{A}(d) \subseteq \mathcal{A}(c)$.
\end{enumerate}
\end{definition}

\begin{proposition}
  The dimension of the cohomology spaces of a vector activation sheaf $\widehat{\mathcal{A}}$ on a link complex $X$ are
  \begin{equation*}
    \text{dim }H^k(\widehat{\mathcal{A}}) = \begin{cases}
      \text{the total number of nodes}&\text{if }k = 0\\
      0&\text{otherwise}\\
      \end{cases}
  \end{equation*}
\end{proposition}
\begin{proof}
  Every global section of $\mathcal{A}$ corresponds to a global section of $\widehat{\mathcal{A}}$, but formal linear combinations of global sections of $\mathcal{A}$ are also global sections of $\widehat{\mathcal{A}}$.  Therefore, a global section of $\widehat{\mathcal{A}}$ merely consists of a list of those nodes that are transmitting, without regard for whether they interfere.

  The fact that the other cohomology spaces are trivial is considerably more subtle.  Consider the decomposition
  \begin{equation*}
    X = \bigcup_i F_i
  \end{equation*}
  of the link complex into the set of its facets.  Suppose that $F_i$ is a facet of dimension $k$, and define $\mathcal{F}_i$ to be the direct sum of $k+1$ copies of the constant sheaf supported on $F_i$.  (Each copy corresponds one of the vertices of $F_i$.)  Then there is an exact sequence of sheaves
  \begin{equation*}
    \xymatrix{
      0 \to \widehat{\mathcal{A}} \ar[r]^\Delta & \bigoplus_i \mathcal{F}_i \ar[r]^m & \mathcal{S} \to 0\\
      }
  \end{equation*}
  where $\Delta$ is a map that takes a basis vector corresponding to a given node to the linear combination of all corresponding basis vectors in each copy of the constant sheaves, and $m$ is therefore a kind of difference map.  This exact sequence leads to a long exact sequence
  \begin{equation*}
\dotsb H^{k-1}(\mathcal{S}) \to H^k(\widehat{\mathcal{A}}) \to \bigoplus_i H^k(\mathcal{F}_i) \to H^k(\mathcal{S}) \dotsb
  \end{equation*}
Since each $\mathcal{F}_i$ is a direct sum of constant sheaves supported on a closed subcomplex, it only has nontrivial cohomology in degree 0.

Observe that $\mathcal{S}$ is a sheaf supported on sets of simplices lying in the intersections of facets.  By Corollary \ref{cor:unaffected}, $\mathcal{S}$ must be a direct sum of copies of constant sheaves supported on closed subcomplexes, like each $\mathcal{F}_i$.  Thus $\mathcal{S}$ only has nontrivial cohomology in degree 0, which means that for $k>1$, $H^k(\widehat{\mathcal{A}}) = 0$.

It therefore remains to address the $k=1$ case, which comes about from the exact sequence
\begin{equation*}
  \bigoplus_i H^0(\mathcal{F}_i) \to H^0(\mathcal{S}) \to H^1(\widehat{\mathcal{S}}) \to 0.
\end{equation*}
The leftmost map is surjective, since every global section of $\mathcal{S}$ is given by specifying a single transmitting node.  By picking exactly one facet containing that node, a global section of the corresponding $\mathcal{F}_i$ may be selected in the preimage.  Thus the map $H^0(\mathcal{S}) \to H^1(\widehat{\mathcal{S}})$ must be the zero map and yet also surjective.  This completes the proof.
\end{proof}

\section{Conclusion}\label{Conclusion}

We have only scratched the surface of topological techniques that can be fruitfully applied to problems in the cyber domain. Discrete Morse theory \cite{Scoville_2019}, the algebraic topology of finite topological spaces \cite{barmak2011algebraic}, and connections between simplicial complexes and partially ordered sets \cite{wachs2007poset} provide just a few opportunities for applications that we have not discussed at all here. For example, a notion of a weighted Dowker complex and an associated partial order can be used for \emph{topological differential testing}\index{topological differential testing} to discover files that similar programs handle inconsistently \cite{ambrose2020topological}. 

More generally, both discrete and continuous topological methods can provide unique capabilities for problems in the cyber domain. The analysis of concurrent protocols and programs highlights this: while simplicial complexes have been used to solve problems in concurrency \cite{herlihy2014distributed}, the entire (recently developed) theory of directed topology traces its origin to static analysis of concurrent programs \cite{fajstrup2016directed}.

In short, while there are many cyber-oriented problems that present a large attack surface for mainstream topological data analysis, the space of applicable techniques is much larger. Cyber problems are likely to continue to motivate future developments in topology, both theoretical and applied.

\subsection*{Acknowledgement}
The authors thank Samir Chowdhury, Fabrizio Romano Genovese, Jelle Herold, and Matvey Yutin for helpful discussions; Greg Sadosuk for producing Figure \ref{fig:intel_cfg} and its attendant code, and Richard Latimer for providing code and data relating to sorting networks.

This research was partially supported with funding from the Defense Advanced Research Projects Agency (DARPA) via Federal contracts HR001115C0050 and HR001119C0072. The views, opinions, and/or findings expressed are those of the authors and should not be interpreted as representing the official views or policies of the Department of Defense or the U.S. Government.

\end{document}